\documentclass[10pt]{article} % default is 10 pt

\usepackage{epstopdf}

\usepackage{latexsym,amsmath,amssymb,a4wide,amscd}
\usepackage[latin1]{inputenc}
\usepackage[T1]{fontenc}

\usepackage{array}

\usepackage{enumerate}

\usepackage{amsthm}

\usepackage{lscape}

\usepackage{url}

\usepackage{mathabx}

\usepackage{longtable}
\usepackage{multirow}

\usepackage{MnSymbol}

\usepackage{amsthm}
\numberwithin{equation}{section}
\numberwithin{figure}{section}
\theoremstyle{plain}
\newtheorem {theorem}{Theorem}[section]

\newtheorem {lemma}[theorem]{Lemma}

\newtheorem {corollary}[theorem]{Corollary}

\theoremstyle{definition}
\newtheorem {definition}[theorem]{Definition}
\theoremstyle{remark}
\newtheorem {remark}[theorem]{Remark}
\newtheorem {observation}[theorem]{Observation}

\usepackage[pdftex]{graphicx} % needed for including graphics e.g. EPS, PS
\usepackage{pdfpages}

\usepackage[nooneline]{caption}

\begin{document}

\date{}

\title{\Large {\bf Random collapsibility and $3$-sphere recognition}}

\author{Jo\~ao Paix\~ao and Jonathan Spreer}

\maketitle

\subsection*{\centering Abstract}

{\em
A triangulation of a $3$-manifold can be shown to be homeomorphic to the
$3$-sphere by describing a discrete Morse function on it with only two critical
faces, that is, a sequence of elementary collapses from the 
triangulation with one tetrahedron removed down to a single vertex.
Unfortunately, deciding whether such a sequence exist is believed to be
very difficult in general.

In this article we present a method, based on uniform spanning trees, to 
estimate how difficult it is to collapse 
a given $3$-sphere triangulation after removing a tetrahedron. In addition
we show that out of all $3$-sphere triangulations with eight vertices or less,
exactly $22$ admit a non-collapsing sequence onto a contractible non-collapsible
$2$-complex. As a side product we classify all minimal triangulations of the 
dunce hat, and all contractible non-collapsible $2$-complexes with at most $18$ 
triangles. This is complemented by large scale experiments on the collapsing
difficulty of $9$- and $10$-vertex spheres. 

Finally, we propose an easy-to-compute
characterisation of $3$-sphere triangulations which experimentally exhibit 
a low proportion of collapsing sequences, leading to a heuristic to produce 
$3$-sphere triangulations with difficult combinatorial properties.}

\medskip
\noindent
\textbf{MSC 2010: } %{\bf 52B12}; % centrally symmetric polytopes
%52B70; % polytopal manifolds
{\bf 57Q15};  % triangulating manifolds
57N12;  % 	Topology of $E^3$ and $S^3$
57M15;   %	Relations with graph theory
%58E05; % Abstract critical point theory (Morse theory, Ljusternik-Schnirelman 
%(Lyusternik-Shnirelʹman) theory, etc.)
%57N10; % general 3-manifolds
%05B10; % difference sets
%20B25; %Finite automorphism groups of algebraic, geometric, or 
%combinatorial structures [See also 05Bxx, 12F10, 20G40, 20H30, 51-XX]
90C59. % approximation and heuristics
%05C81   	Random walks on graphs
%05C05   	Trees
\medskip
\noindent
\textbf{Keywords: } Discrete Morse theory, uniform spanning trees, 
collapsibility, local constructibility, dunce hat, triangulated 
manifolds, 3-sphere, complicated triangulations

\section{Introduction}
\label{sec:intro}

Collapsibility of triangulations and closely related topics, such as local
constructability, have been studied extensively over the past decades
\cite{Benedetti13RandomDMT,Durhuus95LC3Spheres,whitehead1939simplicial,
zeeman1966seminar}. If a triangulation is collapsible, its underlying 
topological space is contractible but the converse is not true
\cite{Adiprasito14RDMTII,Benedetti13NonCollapsible,Benedetti11LocConstrBalls,
Lutz04SmallExNonconstrSimplBallsSph,Zeeman64DunceHat}. Thus, collapsibility can 
be seen as a measure of how complicated a triangulation of a given contractible 
topological space or manifold is. Understanding this
complicatedness of triangulations is a central topic in the field of 
combinatorial topology with important consequences
for theory and applications.

For instance, recognising the $n$-dimensional (piecewise linear standard) 
sphere or ball -- a major challenge in the field -- is still a very difficult 
task in dimension three and even
undecidable in dimensions $\geq 5$ \cite{Volodin743SphereRec}. Nonetheless, 
collapsing heuristics together with standard homology calculations and the 
Poincar\'e conjecture solve the $n$-sphere recognition problem for many 
complicated and large inputs in high dimensions, see for example 
\cite{Benedetti13RandomDMT,Joswig14SphereRec}. In fact, when using standard 
input, existing heuristics are too effective to admit proper insight into the 
undecidability of the underlying problem. To address this issue, 
Benedetti and Lutz recently proposed a ``library of complicated triangulations''
providing challenging input to test new methods \cite{Benedetti13RandomDMT}.

Here we focus on the analysis of collapsibility for 
triangulations of the $3$-ball (typically given as a $3$-sphere with one 
tetrahedron removed). More precisely, we study this question using a 
{\em quantifying} approach. Given a triangulation of the $3$-ball, we
analyse the question of how likely it is that a collapsing sequence of the 
tetrahedra, chosen uniformly at random, collapses down to a single vertex.
A similar question has been studied in \cite{Benedetti13RandomDMT}
using the framework of discrete Morse theory. While the approach in 
\cite{Benedetti13RandomDMT} is efficient, provides valuable insights, and 
works in much more generality (arbitrary dimension and 
arbitrary topology of the triangulations), the probability distributions 
involved in the experiments are difficult to control. As a consequence,
the complicatedness of a given triangulation depends on the heuristic in use.
%In fact, experiments presented in Section~\ref{sec:ust} suggest that 
%the heuristic considered in \cite{Benedetti13RandomDMT} is significantly 
%biased towards finding collapsing sequences .

\medskip
In this article we present an approach to quantify the ``collapsing 
probability'' of a $3$-ball triangulation which can be phrased independently
from the methods in use. This probability can be estimated effectively
as long as there is a sufficient number of collapsing sequences.
We then use this approach to give an extended study of the ``collapsing
probability'' of small $3$-sphere triangulations with one tetrahedron removed.

\medskip
In addition, we decide for all $39$ $8$-vertex $3$-sphere triangulations 
with one tetrahedron removed whether or not they are {\em extendably collapsible}, 
that is, whether or not they have a collapsing sequence of the tetrahedra
which collapses onto a contractible non-collapsible $2$-dimensional 
complex: $17$ of them are, $22$ of them are not, see Theorem~\ref{thm:eight}. 
As a side product of this experiment we present a classification of all minimal
triangulations of the Dunce hat, cf. Theorem~\ref{thm:class}.

\medskip
A major motivation for this study is to find new techniques to tackle the famous 
$3$-sphere recognition problem.
Recognising the $3$-sphere is decidable due to Rubinstein's algorithm 
\cite{Rubinstein953SphereRec} which has since been implemented 
\cite{Burton04Regina,Burton09Regina} and 
optimised by Burton \cite{Burton14Crushing}. 
However, state-of-the art worst case running times are still exponential 
while the problem itself is conjectured to be polynomial time solvable 
\cite{Hass123SphereCoNP,Schleimer11SphereRecNP}. 
We believe that analysing tools -- such as the ones presented in this article -- 
and simplification procedures designed to deal 
with non-collapsible or nearly non-collapsible $3$-sphere triangulations 
(i.e. input with pathological combinatorial features) together with 
local modifications of triangulations such as Pachner moves is one way of 
advancing research dealing with this important question.

\subsection*{Contributions}

In Section~\ref{sec:ust}, we describe a procedure to uniformly sample 
collapsing sequences in $3$-ball triangulations, 
based on the theory of uniform spanning trees 
\cite{Aldous90RandomWalkUST,Broder89RandomWalkUST,
Guenoche83RandomSpanningTrees,Wilson96UST}.

\smallskip
In Section~\ref{sec:estimated}, we present extensive experiments on 
the collapsing probability of small $3$-sphere triangulations with one
tetrahedron removed. The experiments include a complete classification
of extendably collapsible $8$-vertex $3$-spheres with one tetrahedron removed
and a classification of $17$ and $18$ triangle triangulations of 
contractible non-collapsible $2$-complexes.

\smallskip
In Section~\ref{sec:nearlyNonColl} we describe an (experimental) 
hint towards triangulations which are difficult to collapse. 
The observation translates into heuristics to generate complicated 
triangulations.
%
%Finally, in Section~\ref{sec:mfldRec} we present an efficient 
%heuristic procedure to verify the manifold property of a simplicial complex, 
%and discuss further applications 
%of this work involving Heegaard genus estimation of $3$-manifolds and a 
%new heuristic to recognise the un-knot.

\subsection*{Software}

Most of the computer experiments which have been carried out in this article
can be replicated using the {\em GAP} package {\em simpcomp}
\cite{simpcomp,simpcompISSAC,simpcompISSAC11,GAP4}. As of Version 2.1.1., 
{\em simpcomp} contains the functionality to produce {\em discrete Morse
spectra} using the techniques developed in this article as well as the techniques
from \cite{Benedetti13RandomDMT}.

The necessary data to replicate all other experiments can be found in
the appendices and/or are available from the authors upon request.

\section{Preliminaries}
\label{sec:prelim}

\subsection{Triangulations}

Most of this work is carried out in the $3$-dimensional simplicial setting. 
However, whenever obvious generalisations of our results and methods hold 
in higher dimensions, or for more general kinds of triangulations, we 
point this out.

By a {\em triangulated $d$-manifold} (or {\em triangulation of a 
$d$-manifold}) we mean a $d$-dimensional simplicial complex whose 
underlying topological space is a closed $d$-manifold. Note that in dimension
three, the notion of a triangulated $3$-manifold is equivalent to the one
of a combinatorial manifold since every $3$-manifold is equipped with a unique
PL-structure.

A triangulation of a $d$-manifold $M$ is given by a $d$-dimensional, pure,
abstract simplicial complex $C$, i.e., a set of subsets 
$\Delta \subset \{ 1 , \ldots , v \}$ each of order $| \Delta | = d+1$, called 
the {\em facets} of $M$.
The $i$-skeleton $\operatorname{skel}_i (M)$, that is, 
the set of $i$-dimensional faces of $M$ can then be
deduced by enumerating all subsets $\delta$ of order $| \delta | = i+1$ which
occur as a subset of some facet $\Delta \in M$. The $0$-skeleton is called the 
{\em vertices} of $M$, denoted $V(M)$, and the $1$-skeleton is referred to as 
the {\em edges} of $M$. The {\em $f$-vector} of $M$ is defined to be
$f(M) = (f_0, f_1, \ldots , f_d)$ where $f_i = |\operatorname{skel}_i (M)|$.
Note that in this article we often write $f_0 = v$ and $f_d = n$, and use 
$n$ as a measure of input size.

If in a triangulation $M$ every $k$-tuple of vertices spans a $(k-1)$-face
in $\operatorname{skel}_i (M)$, i.e., if $f_{k-1} = {|V(M)| \choose k}$, then 
$M$ is said to be {\em $k$-neighbourly}.

The {\em Hasse diagram} $\mathcal{H} (C)$ of a $d$-dimensional
simplicial complex $C$ is the directed $(d+1)$-partite graph whose nodes are the 
$i$-faces of $C$, $0 \leq i \leq d$, and whose arcs point from 
a node representing an $(i-1)$-face to a node representing an $i$-face if and only 
if the $(i-1)$-face is contained in the $i$-face.

The \textit{dual graph} or {\em face pairing graph} $\Gamma (M)$ 
of a triangulated $d$-manifold $M$ is the graph whose nodes represent the facets of $M$, 
and whose arcs represent pairs of facets of $M$ that are joined 
together along a common $(d-1)$-face. It follows that $\Gamma (M)$ is
$(d+1)$-regular.

\subsection{Uniform spanning trees and random walks}

Most graphs in this article occur as the dual graph $\Gamma(M)$
of some triangulated $3$-manifold $M$. To avoid confusion, we denote
the $0$- and $1$-simplices of a triangulation as vertices and edges and we  
refer to the corresponding elements of a graph as {\em nodes} and {\em arcs}.

A {\em spanning tree} of a graph $G = (V,E)$ is a tree $T = (V,E')$ such that
$E' \subset E$ covers all nodes in $V$. In other words, a spanning tree
of a graph $G$ is defined by a connected subset $E' \subset E$ of size 
$|E'| = |V|-1$ such that all nodes $v \in V$ occur as an endpoint of an 
arc in $E'$. A {\em uniform spanning tree} $T = (V,E')$ of $G=(V,E)$
is a spanning tree chosen uniformly at random from the set of all 
spanning trees of $G$. 

A {\em random walk of length $m$} in a graph $G=(V,E)$ is a sequence of 
random variables $( v_0, v_1, v_2, \ldots , v_m )$ 
taking values in $V$, such that $v_0 \in V$ is 
chosen uniformly at random and for each $v_i$, the vertex $v_{i+1}$ is chosen
uniformly at random from all nodes adjacent to $v_i$ in $G$.

\subsection{Collapsibility and local constructability}

Given a simplicial complex $C$, an $i$-face $\delta \in C$ is called 
{\em free} if its corresponding node in the Hasse diagram $\mathcal{H}(C)$
is of outgoing degree one. Removing a free face $\delta$ of a simplicial complex 
is called an {\em elementary collapse} of $C$, denoted by 
$C \searrow C \setminus \delta$. A simplicial complex $C$ is called
{\em collapsible} if there exist a sequence of elementary collapses 
$$ C \searrow C' \searrow C'' \searrow \ldots \searrow \emptyset , $$
in this case the above sequence is referred to as a {\em collapsing sequence}
of $C$ (sometimes we omit the last elementary collapse from a single 
vertex to the empty set and still refer to the sequence as a collapsing
sequence).

If, for a simplicial complex $C$, {\em every} sequence of removing free faces
leads to a collapsing sequence, $C$ is called {\em extendably collapsible}. If,
on the other hand, no collapsing sequence exist, $C$ is said to be
{\em non-collapsible}.

\medskip
Given a $d$-dimensional simplicial complex $C$, we say that $C$ is {\em locally 
constructible} or that $C$ {\em admits a local construction}, if there is a 
sequence of pure simplicial complexes $T_1, \ldots , T_n, \ldots T_N$ such that 
(i) $T_1$ is a $d$-simplex, (ii) $T_{i+1}$, $i+1 \leq n$, is constructed from 
$T_i$ by gluing a new tetrahedron to $T_i$ along one of its $(d-1)$-dimensional 
boundary faces, (iii) $T_{i+1}$, $i+1 > n$, is constructed from 
$T_i$ by identifying a pair of $(d-1)$ faces of $T_i$ whose intersection 
contains a common $(d-2)$-dimensional face, and (iv) $T_N = C$.

For $d=3$, locally constructible spheres were introduced by Durhuus and Jonsson
in \cite{Durhuus95LC3Spheres}. Locally constructible triangulations of
$3$-spheres are precisely the ones which are collapsible after removing
a facet due to a result by Benedetti and Ziegler 
\cite{Benedetti11LocConstrBalls}.

For the remainder of this article we sometimes call a triangulated 
$3$-sphere $S$ {\em collapsible} if it is locally constructible, i.e., if there 
exist a facet $\Delta \in S$ such that $S \setminus \Delta$ is collapsible.
This notion is independent of the choice of $\Delta$ (cf.
\cite[Corollary 2.11]{Benedetti11LocConstrBalls}). The idea behind this abuse of 
the notion of collapsibility is to refer to those $3$-sphere triangulations as 
collapsible which have a chance of being recognised by a collapsing heuristic.

\section{Collapsibility of $2$-complexes and uniform spanning trees}
\label{sec:ust}

In this section, we want to propose a method to quantify collapsibility 
of $3$-sphere triangulations (with one tetrahedron removed). Deciding 
collapsibility is hard in general 
but easy in most cases which occur in practice and thus methods to measure the
degree to which a triangulation is collapsible are of great help in 
the search for pathological, i.e., non-collapsible $3$-ball triangulations.
The idea is closely related to the concept of the discrete Morse spectrum 
as presented in \cite{Benedetti13RandomDMT},
the main difference being that our method is {\em independent} of the 
collapsing heuristic in use. This, however, comes at the 
cost of only focusing on triangulations of the $3$-sphere 
and possibly slight generalisations thereof.

Our method uses the facts that (i) collapsibility of arbitrary $2$-complexes 
is easy to decide by a linear time greedy type algorithm
\cite[Proposition 5]{Tancer12CollNPComplete}, (ii) spanning trees
of a graph can efficiently be sampled uniformly at random (see below for more
details), and (iii) the process of collapsing the $3$-cells of a $3$-manifold 
triangulation $M$ along a spanning tree of the dual graph $T$ is well defined.
That is, we can collapse all $3$-cells of $M$ along $T$ by first removing the 
$3$-cell $\Delta \in M$ corresponding to the root node of $T$ and then 
successively collapse all other $3$-cells through the $2$-cells of $M$ 
corresponding to the arcs of $T$, and this procedure does not depend on the
choice of $\Delta$, see \cite[Corollary 2.11]{Benedetti11LocConstrBalls}. 

More precisely, for a $3$-sphere triangulation $S$, we can efficiently
sample a spanning tree in the dual graph $T \subset \Gamma (S)$, collapse all 
$3$-cells of $S$ along $T$, and then decide collapsibility of the 
remaining $2$-complex in linear time in the number of facets of $S$. Our method 
leads to the following notion.
\begin{definition}[Collapsing probability]
  \label{def:pcoll}
  Let $S$ be a $3$-sphere triangulation and let $p \in [ 0,1 ]$ 
  be a (rational) number between zero and one. We say that $S$ has 
  {\it collapsing probability $p$} if the number of spanning trees leading to 
  a collapsing sequence of $S$ divided by the total number of 
  spanning trees of $S$ equals $p$.

  In particular, collapsing probability $0$ is equivalent to 
  non-collapsibility and collapsing probability $1$ is equivalent to
  extendable collapsibility.
\end{definition}
The above definition corresponds to a shortened version of what 
is called the {\em discrete Morse spectrum} in \cite{Benedetti13RandomDMT}.

\medskip
Given the notion of collapsing probability, the potential difficulty of deciding 
collapsibility for a $3$-sphere triangulation $S$ (with one tetrahedron 
removed) must be entirely encapsulated within its
extremely large number of possible spanning trees: if this number were small, 
we could simply try all spanning trees of $S$ until we
either find a collapsing sequence or conclude that $S$ 
(with one tetrahedron removed) is non-collapsible. 
But how many spanning trees of $\Gamma (S)$ exist?

The dual graph of any $3$-manifold triangulation is $4$-regular. 
Hence, following \cite{McKay83SpanningTreesRegGraphs} 
the number of spanning trees of a $3$-sphere triangulation 
with $n$ tetrahedra is bounded above by
$$ \# \textrm{ spanning trees } 
  \,\, < \,\, C \cdot \left ( \frac{27}{8} \right )^n \frac{\log n}{n} 
  \,\, < \,\, \frac{9}{2} \left ( \frac{27}{8} \right )^n. $$
%To verify the quality of this upper bound see Example~\ref{ex:cncm}. 
Thus, enumerating spanning trees to decide collapsibility
does not seem like a viable option.

\medskip
However, the related task of sampling a spanning tree uniformly at 
random is efficiently solvable.
The first polynomial time algorithm to uniformly sample spanning trees in an 
arbitrary graph was presented by
Gu\'enoche in 1983 \cite{Guenoche83RandomSpanningTrees}. It has a running 
time of $O(n^3 m)$ where $n$ is the number of nodes and $m$ 
is the number of arcs of the graph. Hence, in the case of the $4$-valent 
dual graph of a triangulated $3$-manifold, 
the running time of Gu\'enoche's algorithm is $O(n^4)$. 
Since then many more deterministic algorithms were constructed with 
considerably faster running times.
Here, we want to consider a randomised approach. Randomised sampling algorithms
for spanning trees were first presented by Broder 
\cite{Broder89RandomWalkUST} and Aldous \cite{Aldous90RandomWalkUST}.
Their approach is based on a simple idea using random walks. Given a graph 
$G$, follow a random walk in $G$ until all nodes have been visited 
discarding all arcs on the way 
which close a cycle. The result can be shown to be a spanning tree 
chosen uniformly at random amongst {\em all} spanning trees of the graph.
The expected running time equals what is called the {\em cover time} of 
$G$, i.e., the expected time it takes 
a random walk to visit all nodes in $G$, with a worst case 
expected running time of $O(n^3)$. 
For many graphs, however, the expected running time is as low as $O(n \log n)$. 
The algorithm we want to use for our
purposes is an improvement of the random walk construction due to Wilson 
\cite{Wilson96UST} which always beats the cover time.
More precisely, the expected running time of Wilson's algorithm is $O(\tau)$ 
where $\tau$ denotes the expected 
number of steps of a random walk until it hits an already determined subtree 
$T' \subset G$ starting from a node
which is not yet covered by $T'$.

\begin{observation}
  Let $S$ be a $3$-sphere triangulation with $n$ tetrahedra and
  collapsing probability $p \in [0,1]$. Sampling 
  a uniform spanning tree in the dual graph of $S$ and testing
  collapsibility of the remaining $2$-complex is a Bernoulli trial 
  $$ X = \left \{ \begin{array}{ll} 1 &\textrm{ with probability } p; \\
    0 &\textrm{ else.} \end{array} \right . $$
  with polynomial running time.
  Sampling $N$ times yields $N$ such independent Bernoulli distributed random 
  variables $X_i$, $1\leq i \leq N$, and the maximum likelihood estimator 
  $$\hat{p} = \frac{1}{N} \sum \limits_{i=1}^{N} X_i $$
  follows a normalised Binomial distribution with $\mathbb{E} \hat{p} = p$ and
  $\operatorname{Var} \hat{p} = p(1-p)/N$.
  By Chebyshev's inequality this translates to
  $$ \mathbb{P} \left ( \, \mid \, \hat{p} \,-\, p \, \mid \, \leq \, \epsilon \right ) \,\, < \,\, \frac{p(1-p)}{N \epsilon^2} \,\, \leq \,\, \frac{1}{4 N \epsilon^2}. $$
  Since we want to decide collapsibility of $S$ (with one tetrahedron removed), 
	we want to distinguish $p$ 
  from $0$. Thus, setting $\epsilon = p/2$ we get
  $$ \mathbb{P} \left ( \, \, \mid \, \hat{p} \,-\, p \, \mid \, \leq \, p/2  \, \right ) < \frac{4 (1-p)}{N p} $$
  and for $p \to 0$ the error can be controlled by $N$ being (super-)linear in $p^{-1}$.

  \medskip
  Altogether, we note that collapsibility of $S$ (with one tetrahedron removed)
  can be rejected with a 
  high level of confidence by a polynomial procedure as long as 
  $p^{-1}$ is polynomial in the size of the input $n$.

\end{observation}

Note that, using Wilson's algorithm, every sampling procedure only depends on 
the size of the triangulation by a factor which, in average, has running
time less than $O(n \log n)$. This makes the approach well-suited for computer 
experiments on larger triangulations where a higher proportion of 
triangulations with a low collapsing rate is conjectured.

\section{Estimated collapsing probabilities for small $3$-sphere
triangulations}
\label{sec:estimated}

A small triangulation of a $3$-sphere admits very few or no
spanning trees in the dual graph, which lead to 
non-collapsing sequences. As the size of
triangulations increase, we expect that the number of 
non-collapsible sequences increases as well. See \cite{Joswig14SphereRec}
for experiments supporting this claim in higher dimensions.
However, given a sequence of somewhat ``averagely complicated'' $3$-sphere 
triangulations in increasing size, the question of how exactly the 
{\em proportion} of collapsing sequences to non-collapsing sequences changes, 
that is, how quickly (if at all) the collapsing probability $p$ decreases, 
is an interesting question with deep implications for important problems
in the field of computational $3$-manifold topology. One major difficulty in 
this context comes from the fact that it is not at all clear what is meant by an
``averagely complicated'' $3$-sphere triangulation.

\medskip
In this section, we use the method of uniformly sampling spanning trees
described above, together with the classification of $3$-sphere triangulations
up to ten vertices to get a more thorough overview in the case of known small 
triangulations.

\subsubsection*{$3$-spheres with up to $8$-vertices}

It is well-known that all $3$-balls with seven or less vertices are extendably
collapsible \cite{Bagchi08UniqWalkup9V3DimKleinBottle}. 
Hence, all $v$-vertex $3$-sphere triangulations with $v \leq 7$ must have 
collapsing probability $1$.

There are $39$ distinct $8$-vertex triangulations of the $3$-sphere, first 
listed by Gr\"unbaum and Sreedharan in \cite{Gruenbaum67Enum8vtxSpheres}.
We use their notation for the remainder of this section. Two of them are 
non-polytopal, one of them known as {\em Barnette's sphere} 
\cite{Barnette69BarnettesSphere}, the other one known as {\em Br\"uckner's sphere} 
(see \cite{Gruenbaum67Enum8vtxSpheres} where the sphere, first described in 
\cite{Brueckner09FourPolytope}, is first shown to be non-polytopal). 

The collapsibility of these $3$-sphere triangulations 
(with one tetrahedron removed) was studied in \cite{Benedetti13NonExtColl3Ball},
where the authors showed that three of the $39$ spheres contain a collapsing 
sequence onto a dunce hat and hence have collapsing probability $< 1$. 

Here, we refine this study by a complete description of $8$-vertex
$3$-spheres with collapsing probability $1$ and a list of non-perfect
Morse functions for the remaining cases. Namely, we  
present a computer assisted proof of the following statement.

\begin{theorem}
  \label{thm:eight}
  There are $17$ $8$-vertex triangulations of the $3$-sphere which, after
  removing any tetrahedron, are extendably collapsible. In the Gr\"unbaum,
  Sreedharan notation these are
  $$ P_{1}, P_{2}, P_{3}, P_{4}, P_{5}, P_{6}, P_{7}, P_{9}, P_{10}, P_{13}, P_{14}, P_{15}, P_{16}, P_{17}, P_{18}, P_{21}, P_{34}. $$
  The remaining $22$ $8$-vertex triangulations of the $3$-sphere triangulations
  admit a collapsing sequence onto a contractible non-collapsible 
  $2$-complex.
  %$P_{8}$, $P_{11}$, $P_{12}$, $P_{19}$, 
  %$P_{20}$, $P_{22}$, $P_{23}$, $P_{24}$, 
  %$P_{25}$, $P_{26}$, $P_{27}$, $P_{28}$, $P_{29}$, 
  %$P_{30}$, $P_{31}$, $P_{32}$, $P_{33}$, $P_{35}$, $P_{36}$, $P_{37}$,   
  %$M$,  $B$, 
\end{theorem}

In order to construct a computer assisted proof for Theorem~\ref{thm:eight}, 
we first need to make some theoretical observations.

\begin{lemma}
  \label{lem:2compl}
  Let $C$ be a contractible non-collapsible simplicial $2$-complex. Then $C$ 
  must have at least $8$ vertices and $17$ triangles.
\end{lemma}

\begin{proof}
  Let $C$ be a contractible, non-collapsible $2$-complex of minimal size.
  Since $C$ is contractible, we must have vanishing $2$-homology.
  In particular $H_2 (C, \mathbb{F}_2) = 0$, where $\mathbb{F}_2$
  is the field with two elements. Look at the formal sum $\sigma$ of all 
  triangles of $C$. $\sigma$ is a $2$-chain and its boundary $\partial \sigma$ 
  contains all edges of $C$ of odd degree.

  If $C$ has no edges of odd degree then $\sigma$ is a non-vanishing
  element of $H_2 (C, \mathbb{F}_2)$, contradiction.
  Hence, $C$ must have edges of odd degree. In addition, since the edges of odd 
  degree are a boundary, they must form a closed cycle and since $C$ is 
  simplicial, this cycle must be of length at least three.

  Since $C$ is minimal non-collapsible, no edge can be of degree one, and 
  altogether all edges must be of degree at least two and at least three edges 
  must be of degree at least three.

  \medskip
  Let $f(C) = (n, f_1, f_2)$ be the $f$-vector of $C$. 
  Since $C$ is contractible, $C$ has Euler characteristic one.
  Hence
  $$ \begin{array}{ccc}
    1 &=& n - f_1 + f_2 \\
    f_2 &=& \frac13 \sum \limits_{e \textrm{ edge of } C} \operatorname{deg}(e),
  \end{array} $$
  and $f_2$ is minimal if and only if three edges are of degree three and
  all other edges are of degree two. Inserting these degrees into the
  second equation yields
  $$ f_2 \geq \frac23 f_1 + 1 $$
  and using the first equation we get
  $$ f_2 \geq 2n +1. $$ 
  Following the results in \cite{Bagchi05CombTrigHomSpheres} we have $n \geq 8$
  and thus $C$ must have at least $17$ triangles.
  %
  %\medskip
  %Suppose $C$ has exactly $17$ triangles, thus $8$ vertices, $21$ edges of
  %degree two and $3$ edges of degree three which form a cycle. Ungluing this 
  %cycle yields a non-singular surface $C'$ with $17$ triangles,
  %$30$ edges, $9$ of which are boundary edges, and $14$ vertices, $9$ of 
  %which lie in the boundary. Hence, $C'$ has Euler characteristic
  %$1$ and must be a disk. To see that $C'$ is non-singular recall that $C$ is 
  %contractible.
  %
  %Altogether, we can represent $C$ as a triangle, each side represented by
  %three boundary edges, where all three sides are glued together to form a 
  %contractible complex. Hence, $C$ must be a dunce hat.
\end{proof}

\begin{corollary}
  \label{cor:extcoll}
  Let $S$ be a $3$-sphere triangulation with $f$-vector
  $f(S) = (v, v+n, 2n, n)$, and let $B$ be a $3$-ball obtained from
  $S$ by removing a tetrahedron. If $v < 8$ or $n < 16$ then 
  $B$ is extendably collapsible.  
\end{corollary}

\begin{proof}
  A $3$-ball $B$ is extendably collapsible if, after removing its $3$-cells
  along a spanning tree of the dual graph, the remaining contractible 
  $2$-complex must be collapsible.

  After collapsing the tetrahedra of $S$ along a spanning tree the remaining
  $2$-complex has $(n+1)$ triangles and $v$ vertices. The result now follows
  from Lemma~\ref{lem:2compl}. 
\end{proof}

The observation made in Corollary~\ref{cor:extcoll} can be extended to 
$3$-sphere triangulations $S$ with a larger number $n \geq 16$ of tetrahedra. 
If we can show that the $2$-skeleton of $S$ does not contain a contractible 
non-collapsible $2$-complex with at most $n+1$ triangles, then $S$ (with one
tetrahedron removed) must
be extendably collapsible. In order for this approach to work, we need to 
know all such $2$-complexes up to $n+1$ triangles. Like any other attempt
to exhaustively classify small triangulations, this rapidly becomes infeasible
with $n$ growing larger. However, in the border case of $16 \leq n \leq 17$
this task turns out to be well in reach.

\begin{theorem}
  \label{thm:class}
  The only non-collapsible contractible simplicial complexes
  with $8$ vertices and $17$ triangles are the seven minimal triangulations
  of the dunce hat shown in Figure~\ref{fig:dH}.
  Furthermore, the only non-collapsible contractible simplicial complexes
  with $8$ vertices and $18$ triangles are the $19$ minimal saw-blade complexes
  with four, three, and two blades shown in Figures~\ref{fig:sb4}, 
  \ref{fig:sb3}, and \ref{fig:sb2} respectively, and the $61$ triangulations
  of the Dunce hat listed in Appendix~\ref{app:18}.
\end{theorem}

\begin{proof}
  %By the proof of Lemma~\ref{lem:2compl},
  Since $n \leq 18$, no edges of degree other than two and three can exist,
  and the edges of degree three must form a simple cycle.
  By ungluing the edges of degree three, an $8$-vertex, $17$-triangle
  non-collapsible contractible $2$-complex can thus be represented as a 
  $14$-vertex triangulation of the disk with a nine-gon as boundary plus
  some boundary identifications. Analogously, an $8$-vertex, $18$-triangle 
  non-collapsible contractible $2$-complex can be represented as a $16$-vertex 
  triangulation of the disk with a $12$-gon as boundary and some boundary
  identifications.

  \begin{center}
    \begin{figure}[htb]
      \includegraphics[width=\textwidth]{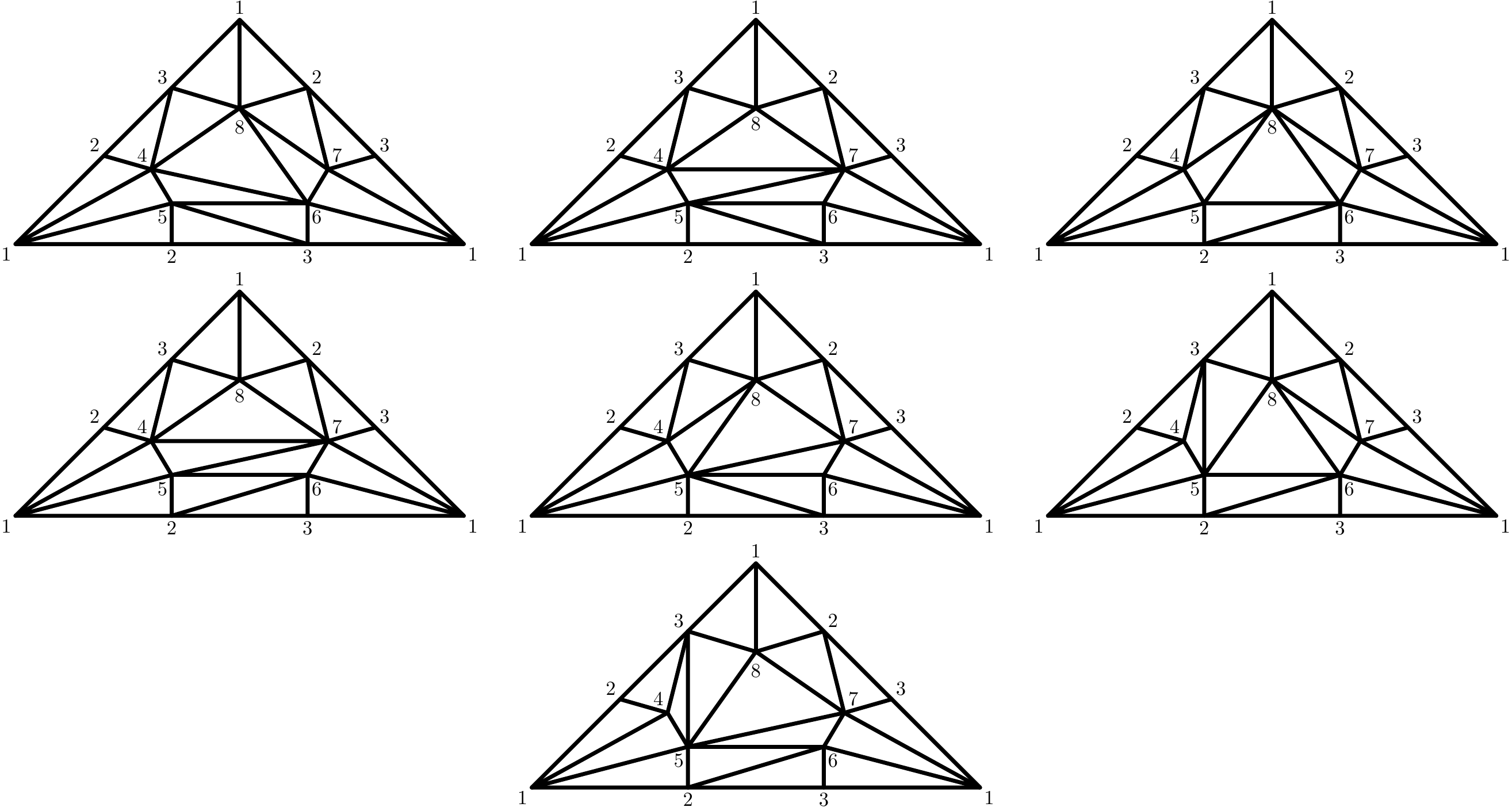}
      \caption{A classification of minimal triangulations of the Dunce hat.
        \label{fig:dH}}
    \end{figure}
  \end{center}

  Using the software for planar graph enumeration by Brinkmann and McKay 
  \cite{Brinkmann07FastGenPlanarGraphs,plantri} together with its enormously 
  useful plug-in framework, we classify all such disks, only keeping those
  which can be folded up to result in an $8$-vertex $17$ ($18$) triangle 
  contractible non-collapsible simplicial complex with $21$ degree two and 
  $3$ ($4$) degree three edges. Sorting out isomorphic copies yields
  seven minimal triangulations of the dunce hat (drawn in Figure~\ref{fig:dH}),
  $61$ dunce hats with $18$ triangles, and $19$ examples of three distinct 
  types of so-called {\em saw-blade complexes}, cf. \cite{Joswig14SphereRec}
  (see Figures~\ref{fig:sb4}, \ref{fig:sb3}, and \ref{fig:sb2} for the $19$ saw 
  blade complexes and Appendix~\ref{app:18} for a list of all $18$-triangle
  complexes).
\end{proof}

  \begin{center}
    \begin{figure}[htb]
      \centerline{\includegraphics[width=0.7\textwidth]{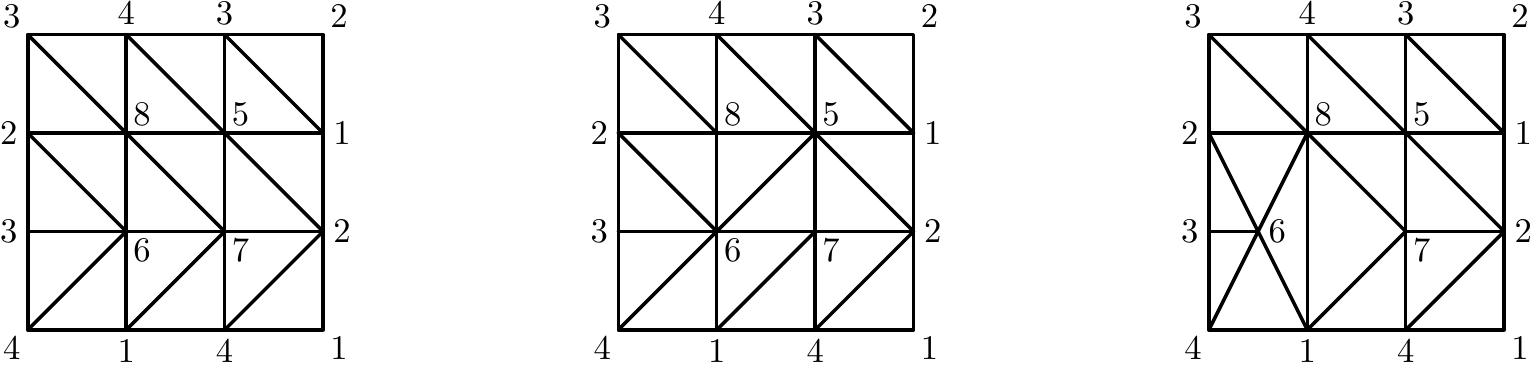}}
      \caption{A classification of minimal saw-blade complexes with four blades.
        \label{fig:sb4}}
    \end{figure}
  \end{center}

With these results in place we can now describe a computer assisted proof
of Theorem~\ref{thm:eight}.

  \begin{center}
    \begin{figure}[htb]
      \centerline{\includegraphics[width=0.435\textwidth]{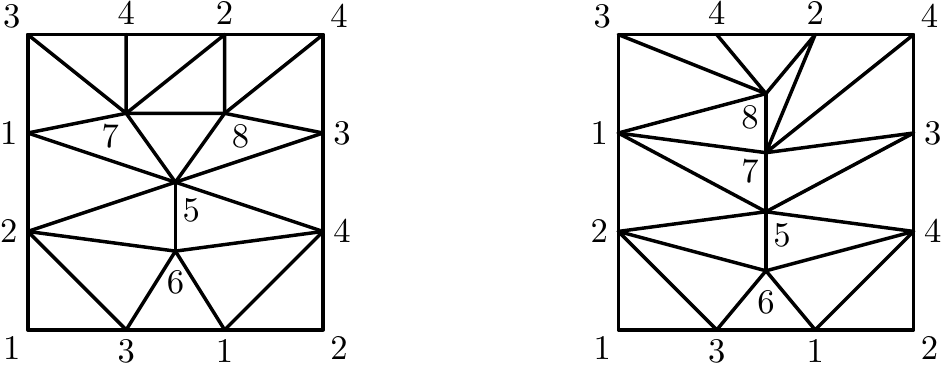}}
      \caption{A classification of minimal saw-blade complexes with three
        blades. \label{fig:sb3}}
    \end{figure}
  \end{center}

\begin{proof}[Proof of \ref{thm:eight}]
  Eight of the $39$ $8$-vertex $3$-spheres have $15$ or less 
  tetrahedra and thus, by Corollary~\ref{cor:extcoll}, must have collapsing 
  probability $1$ (i.e., they are extendably collapsible after removing a 
  tetrahedron). In Gr\"unbaum and Sreedharan's labelling these are the 
  triangulations $P_{1}$ to $P_{7}$, and~$P_{13}$.

  Furthermore, using our uniform sampling technique of spanning trees we
  are able to collapse $22$ of the remaining $31$ triangulations to a 
  contractible non-collapsible $2$-complex and thus show that these have 
  collapsing probability $<1$, i.e.,
  that none of them, after removing a tetrahedron, is extendably collapsible.
  A certificate for the non-extendable collapsibility for each of the $22$ 
  cases, in form of a non-perfect Morse function, can be found in 
  Appendix~\ref{app:eightV}. 

  The remaining nine cases, triangulations $P_{9}, P_{10}, P_{14}, \ldots , 
  P_{18}, P_{21}, P_{34}$ in \cite{Gruenbaum67Enum8vtxSpheres}, have between 
  $16$ and $17$ tetrahedra. 
  Following the proof of Corollary~\ref{cor:extcoll}, after collapsing the 
  tetrahedra of an $n$-tetrahedra $8$-vertex $3$-sphere $S$ along a spanning 
  tree the remaining $2$-complex $C$, which by construction must be 
  contractible, has $(n+1)$ triangles and at most $8$ vertices. Combining 
  Lemma~\ref{lem:2compl} and Theorem~\ref{thm:class} this means that $C$ either 
  collapses onto a point or it is isomorphic to one of the seven minimal 
  triangulations of the dunce hat, for $16 \leq n \leq 17$, or it is isomorphic
  to one of the $80$ contractible non-collapsible $2$-complexes with $18$ 
  triangles, in the case of $n=17$ only.

  \begin{center}
    \begin{figure}[htb]
      \centerline{\includegraphics[width=\textwidth]{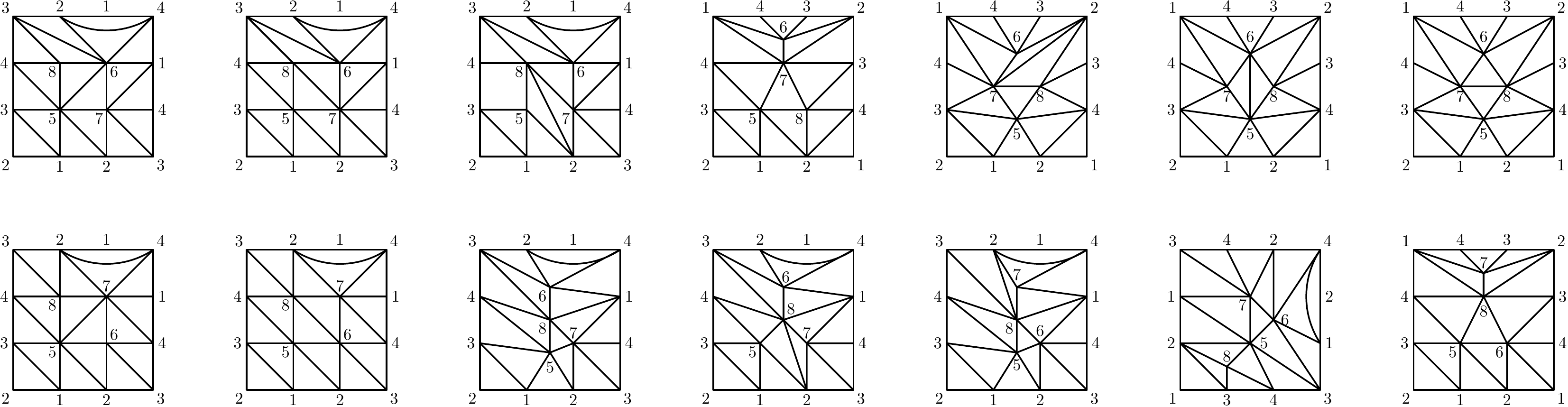}}
      \caption{A classification of minimal saw-blade complexes with two blades.
        \label{fig:sb2}}
    \end{figure}
  \end{center}

  An exhaustive search for all labellings of all $87$ complexes in the
  $2$-skeleta of all nine $3$-spheres showed that none of the nine remaining
  $8$-vertex $3$-spheres contain such a contractible non-collapsing
  $2$-complex. Thus all of them are, after removing a tetrahedron,
  extendably collapsible.

\end{proof}

In order to complement the qualitative study on collapsibility given by 
Theorem~\ref{thm:eight}, we applied our uniform spanning tree heuristic to the 
$19$ and $20$ tetrahedra $8$-vertex $3$-spheres using $10^7$ samples for each 
complex. This gives a more detailed estimate for the collapsing probability
in the $8$-vertex case.

The results are listed in Table~\ref{tab:eight} together with their edge
variance, which will be discussed as indicator for the collapsing probability
in Section~\ref{sec:nearlyNonColl}.

\begin{center}
  \begin{longtable}{|l|l|r|r|}
    \caption{Estimated collapsing probabilities and edge variance of  
    $8$-vertex $3$-sphere triangulations with 
    $19$ and $20$ tetrahedra, sample size 
    $N = 10^7$. \label{tab:eight}} \\
    
    \hline
    name in \cite{Gruenbaum67Enum8vtxSpheres}&$f$-vector / isomorphism signature*&coll. prob.&edge var. \\
    \hline
    \hline
    \endfirsthead

    \multicolumn{4}{l}%
    { \tablename\ \thetable{} -- continued from previous page} \\
    \hline
    name in \cite{Gruenbaum67Enum8vtxSpheres}&$f$-vector / isomorphism signature*&coll. prob.&edge var. \\
    \endhead

    \hline \multicolumn{4}{r}{{continued on next page --}} \\
    \endfoot

    \hline
    \endlastfoot
%    $(8,26,36,18)$&&\\
%    \hline
%    deeffaf.gbh.gbhbg.gbh.hchcgggch&$1/50\,000$&$4.1694$\\
%    deeffaf.gbh.gbhbg.hbh.gggdhehgg&$0$&$4.1694$\\
%    deeffaf.fbg.hbgcgbhchch.hdgahch&$0$&$4.6243$\\
%    deeffaf.gbg.gbhbg.hbg.ghhchagch&$0$&$4.6243$\\
%    deeffaf.gbh.hbhbg.gbh.gbhigbgjg&$0$&$4.6243$\\
%    deeffaf.fbg.gbgchbhchchchahah.h&$0$&$5.0383$\\
%    \hline
    &\multicolumn{3}{l|}{$(8,27,38,19)$} \\
    \hline
    $B$      &\texttt{deeffag.hbg.hag.hbg.hchdgbh.hehgh}** &$99.99902 \%$ &$0.83951$\\
    $P_{32}$ &\texttt{deeffaf.gbg.gbgbh.gbh.hch.hahghcg}   &$99.99922 \%$ &$0.83951$\\
    $P_{31}$ &\texttt{deeffaf.gbg.gbgbh.hbg.hch.hghbhjh}   &$99.99984 \%$ &$0.98765$\\
    $P_{30}$ &\texttt{deeffaf.gbg.gbhbg.hbh.gfhahchehgg}   &$99.99990 \%$ &$0.98765$\\
    $P_{33}$ &\texttt{deeffaf.gbh.hbhbg.gbg.gbhcgegag.g}   &$99.99988 \%$ &$1.13580$\\
    \hline
    &\multicolumn{3}{l|}{$(8,28,40,20)$} \\
    \hline
    $M$      &\texttt{deeffaf.gbh.gbgbh.gbh.hch.hbg.gehcg}*** &$99.99657 \%$ &$0.91837$\\
    $P_{37}$ &\texttt{deeffaf.gbg.gbgbh.hbh.hch.hahehah.h}    &$99.99942 \%$ &$1.20408$\\
    $P_{36}$ &\texttt{deeffaf.gbg.gbhbg.hbh.hchbhahcheh.h}    &$99.99968 \%$ &$1.20408$\\
    $P_{35}$ &\texttt{deeffaf.gbg.gbhbh.hbh.gdhahahchfgfg}    &$99.99997 \%$ &$1.34694$\\
    \hline
  \end{longtable}
\end{center}

\small
\noindent
* The isomorphism signature of a combinatorial manifold uniquely determines its isomorphism
type, i.e., two combinatorial manifolds have equal isomorphism signature if and only if
they are isomorphic. The isomorphism signature given in this table coincides with the one used
by simpcomp \cite{simpcomp,simpcompISSAC,simpcompISSAC11}. Use the function
\texttt{SCFromIsoSig(...)} to generate the complexes. See the manual for details.

\smallskip
\noindent
** This is Barnette's sphere.

\smallskip
\noindent
*** This is Br\"uckner's sphere.
\normalsize

\subsubsection*{$9$-vertex spheres}

There are $1296$ triangulated $9$-vertex $3$-spheres first described
in \cite{Altshuler76CombMnf9VertAll}. We sampled $5\cdot 10^4$ spanning trees
for each of them. The results are summarised below where all $3$-sphere
triangulations with equal number of tetrahedra $n$ are grouped together.
Similar to the $8$-vertex case, a higher number of tetrahedra 
correlates with a lower collapsing probability. The rightmost column
shows the complex with the smallest collapsing probability for each 
class of triangulations. Note that some of the empirical collapsing 
probabilities below are too close to $1$ in order to give robust estimators.

\newpage
\begin{center}
  \begin{longtable}{|r|r|r|r|}
    \caption{Estimated collapsing probabilities of $9$-vertex $3$-sphere 
    triangulations grouped by $f$-vector, and minimal estimated(!) collapsing 
    probability per group, sample size $N = 5\cdot 10^4$ per complex. \label{tab:nine}} \\
    
    \hline
    $n$&$\#$ complexes&avg. coll. prob.&min. coll. prob. \\
    \hline
    \hline
    \endfirsthead

    \multicolumn{4}{l}%
    { \tablename\ \thetable{} -- continued from previous page} \\
    \hline
        $n$&$\#$ complexes&avg. coll. prob.&min. coll. prob. \\
    \endhead

    \hline \multicolumn{4}{r}{{continued on next page --}} \\
    \endfoot

    \hline
    \endlastfoot
    $17$&  $7$&$100.00000 \%$ &$100.000 \%$\\
    \hline
    $18$& $23$&$100.00000 \%$ &$100.000 \%$\\
    \hline
    $19$& $45$&$100.00000 \%$ &$100.000 \%$\\
    \hline
    $20$& $84$&$99.99993 \%$  &$99.998 \%$\\
    \hline
    $21$&$128$&$99.99983 \%$  &$99.998 \%$\\
    \hline
    $22$&$175$&$99.99952 \%$  &$99.996 \%$\\
    \hline
    $23$&$223$&$99.99898 \%$  &$99.994 \%$\\
    \hline
    $24$&$231$&$99.99753 \%$  &$99.980 \%$\\
    \hline
    $25$&$209$&$99.99443 \%$  &$99.962 \%$\\
    \hline
    $26$&$121$&$99.99051 \%$  &$99.952 \%$\\
    \hline
    $27$& $50$&$99.98024 \%$  &$99.920 \%$\\
    \hline
  \end{longtable}
\end{center}

In Section~\ref{sec:nearlyNonColl} below we discuss how the square of the 
average difference between an edge degree and the average edge degree of a 
triangulation, the {\em edge variance} (cf. Definition~\ref{def:edgevar}), influences
the collapsing probability of a triangulation. Essentially, the findings of
Section~\ref{sec:nearlyNonColl} suggest that a smaller edge variance correlates 
with a lower collapsing probability. The following table lists empirical 
collapsing probabilities for the triangulation with minimum edge variance for 
each class of triangulations with fixed $f$-vector with a much higher number
of $10^6$ samples per complex. Compare the estimated collapsing probabilities 
with the values from the table above and with the values
given for the $8$- and $10$-vertex case.     

\begin{center}
  \begin{longtable}{|r|l|@{\,}l@{\,}l@{\,}l@{\,}l@{\,}l@{\,  }|r|}
    \caption{Estimated collapsing probabilities of $9$-vertex, $n$-tetrahedron
    $3$-sphere triangulations with minimum edge variance, $23\leq n \leq 27$,
    sample size $N = 10^6$ per complex. \label{tab:ninemin}} \\
    
    \hline
    $n$&isomorphism signature*&\multicolumn{5}{l|}{edge degrees}&coll. prob. \\
    \hline
    \hline
    \endfirsthead

    \multicolumn{8}{l}%
    { \tablename\ \thetable{} -- continued from previous page} \\
    \hline
    $n$&isomorphism signature*&\multicolumn{5}{l|}{edge degrees}&coll. prob. \\
    \endhead

    \hline \multicolumn{8}{r}{{continued on next page --}} \\
    \endfoot

    \hline
    \endlastfoot
    %$17$& \texttt{deeefbfcfcgcgbgchcichbheiah.h}                     & &$3^{11}$&$4^{8}$&$5^{5}$&$6^{2}$     &$100.0000 \%$\\
    %\hline
    %$18$& \texttt{deeefbfcgchbg.hag.hcicicicibici}                   & &$3^{8}$&$4^{12}$&$5^{6}$&$6^{1}$       &$100.0000 \%$\\
    %\hline
    %$19$& \texttt{deeefbgchciag.h.iah.icifh.icifiki}                 & &&$3^{4}$&$4^{18}$&$5^{6}$            &$99.9999 \%$\\
    %\hline
    %$20$& \texttt{defgh.f.g.hbg.hciciag.hcicicicijici}               & &&&$4^{25}$&$5^{4}$                 &$100.0000 \%$\\
    %\hline
    %$21$& \texttt{deeffag.gbh.iag.hbg.ibhcidifhdibigiki}             & &&$3^{3}$&$4^{18}$&$5^{9}$            &$100.0000 \%$\\
    %\hline
    %$22$& \texttt{deeffag.hbg.iag.ibh.ichbidi.iei.hfigioh}           & &$3^{4}$&$4^{16}$&$5^{10}$&$6^{1}$      &$99.9994 \%$\\
    %\hline
    $23$& \texttt{deeffag.hbg.iag.ibh.ichbidh.ibhbi.hfipijh}         & &$3^{5}$&$4^{14}$&$5^{11}$&$6^{2}$      &$99.9981 \%$\\
    \hline
    $24$& \texttt{deeffaf.gbh.gbgbi.gbh.hch.ibg.geicg.hgigiji}       & &$3^{6}$&$4^{13}$&$5^{10}$&$6^{4}$      &$99.9912 \%$\\
    \hline
    $25$& \texttt{deeffaf.gbh.gbgbi.gbh.ici.ibi.gbibhchchcikhdg}     & &$3^{6}$&$4^{12}$&$5^{12}$&$6^{4}$      &$99.9766 \%$\\
    \hline
    $26$& \texttt{deeffaf.gbh.gbgbh.gbh.ici.hbi.ibibhaiag.ifihhjg}   & $3^{6}$&$4^{13}$&$5^{11}$&$6^{4}$&$7^{1}$ &$99.9485 \%$\\
    \hline
    $27$& \texttt{deeffaf.gbh.gbgbi.gbh.ici.hbi.ibibhaiahahahcihhjg} & &$3^{6}$&$4^{12}$&$5^{15}$&$7^{3}$      &$99.9007 \%$\\
    \hline
  \end{longtable}
\end{center}

The full list of complexes and their empirical collapsing probabilities for
$5\cdot 10^4$ samples is available from the authors upon request.

\subsubsection*{$10$-vertex spheres}

There are $247\,882$ triangulated $9$-vertex $3$-spheres first described
in \cite{Lutz08ThreeMfldsWith10Vertices}. We sampled $5\cdot 10^3$ spanning trees
for each of them. The results grouped by number of tetrahedra are summarised 
in the table below.

\begin{center}
  \begin{longtable}{|r|r|r|r|}
    \caption{Estimated collapsing probabilities of $10$-vertex $3$-sphere triangulations 
    grouped by $f$-vector, and minimal estimated(!) collapsing probability per group,
    sample size $N = 5\cdot 10^3$ per complex. \label{tab:ten}} \\
    
    \hline
    $n$&$\#$ complexes&avg. coll. prob.&min. coll. prob. \\
    \hline
    \hline
    \endfirsthead

    \multicolumn{4}{l}%
    { \tablename\ \thetable{} -- continued from previous page} \\
    \hline
    $n$&$\#$ complexes&avg. coll. prob.&min. coll. prob. \\
    \endhead

    \hline \multicolumn{4}{r}{{continued on next page --}} \\
    \endfoot

    \hline
    \endlastfoot
    $20$&   $30$&$100.00000 \%$&       $100.00 \%$\\
    \hline
    $21$&  $124$&$100.00000 \%$&       $100.00 \%$\\
    \hline
    $22$&  $385$&$99.99990 \%$&$99.98 \%$\\
    \hline
    $23$&  $952$&$99.99989 \%$&$99.98 \%$\\
    \hline
    $24$& $2142$&$99.99966 \%$&$99.96 \%$\\
    \hline
    $25$& $4340$&$99.99936 \%$&$99.96 \%$\\
    \hline
    $26$& $8106$&$99.99860 \%$&$99.94 \%$\\
    \hline
    $27$&$13853$&$99.99750 \%$&$99.92 \%$\\
    \hline
    $28$&$21702$&$99.99521 \%$&$99.90 \%$\\
    \hline
    $29$&$30526$&$99.99144 \%$&$99.86 \%$\\
    \hline
    $30$&$38553$&$99.98578 \%$&$99.80 \%$\\
    \hline
    $31$&$42498$&$99.97656 \%$&$99.72 \%$\\
    \hline
    $32$&$39299$&$99.96899 \%$&$99.52 \%$\\
    \hline
    $33$&$28087$&$99.94089 \%$&$99.40 \%$\\
    \hline
    $34$&$13745$&$99.91159 \%$&$99.16 \%$\\
    \hline
    $35$& $3540$&$99.87571 \%$&$99.20 \%$\\
    \hline
  \end{longtable}
\end{center}

Again, for each number of tetrahedra, we ran $10^6$ samples on the
triangulation with minimum edge variance. 

\begin{center}
  \begin{longtable}{|r|l|@{\,}l@{\,}l@{\,}l@{\,}l@{\,}l@{\,  }|r|}
    \caption{Estimated collapsing probability of $10$-vertex, $n$-tetrahedron
    $3$-sphere triangulation with minimum edge variance, $27\leq n \leq 35$,
    sample size $N = 10^6$ per complex. \label{tab:tenmin}} \\
    
    \hline
    $n$&isomorphism signature*&\multicolumn{5}{l|}{edge degrees}&coll. prob. \\
    \hline
    \hline
    \endfirsthead

    \multicolumn{8}{l}%
    { \tablename\ \thetable{} -- continued from previous page} \\
    \hline
    $n$&isomorphism signature*&\multicolumn{5}{l|}{edge degrees}&coll. prob. \\
    \endhead

    \hline \multicolumn{8}{r}{{continued on next page --}} \\
    \endfoot

    \hline
    \endlastfoot

    %$20$& {\footnotesize \texttt{deeefbfcfcgcgbgchchchbicicicjcjcjcj}}&
    %  &$3^{12}$&$4^{9}$&$5^{6}$&$6^{3}$&$100.0000 \%$\\
    %\hline
    %$21$& {\footnotesize \texttt{deeefbfcfcgcgbhcibh.ibh.ibjcjcjcjcjcj}}&
    %  &$3^{9}$&$4^{14}$&$5^{5}$&$6^{3}$&$100.0000 \%$\\
    %\hline
    %$22$& {\footnotesize \texttt{deefgaf.gbf.gchchbhchcicjcicjbicjejfjgj}}&
    %  &$3^{6}$&$4^{17}$&$5^{8}$&$6^{1}$&$99.9999 \%$\\
    %\hline
    %$23$& {\footnotesize \texttt{deefgaf.gbf.gchchbicicjcjbi.jbi.jbjcjkjgj}}&
    %  &&$3^{3}$&$4^{21}$&$5^{9}$&$99.9998 \%$\\
    %\hline
    %$24$& {\footnotesize \texttt{defgh.f.g.hbg.hciciag.hcicicjcjcjcjbjcjcjcj}}&
    %  &&&$4^{26}$&$5^{8}$&$99.9998 \%$\\
    %\hline
    %$25$& {\footnotesize \texttt{defgh.f.g.ibj.hbj.haicj.ibj.ighbidhbigigigigi}}&
    %  &&&$4^{25}$&$5^{10}$&$99.9999 \%$\\
    %\hline
    %$26$& {\footnotesize \texttt{deeffag.hbi.jag.ibh.jbi.hbjdjfi.hcjbjghgigjcjkj}}&
    %  &&$3^{4}$&$4^{16}$&$5^{16}$&$99.9992 \%$\\
    %\hline
    $27$& {\footnotesize \texttt{deeffaf.gbh.gbibj.ibh.jcj.jbg.iciahcigjghcidijjki}}&
      &$3^{4}$&$4^{18}$&$5^{12}$&$6^{3}$&$99.9974 \%$\\
    \hline
    $28$& {\footnotesize \texttt{deeffaf.gbh.gbibj.ibh.hch.jbg.iciajci.hgjgjbidibjwi}}&
      &$3^{4}$&$4^{19}$&$5^{10}$&$6^{5}$&$99.9917 \%$\\
    \hline
    $29$& {\footnotesize \texttt{deefgaf.hbg.hbi.iai.j.iaj.hdiaj.j.ibj.jajchcjkhgjdirh}}&
      &$3^{5}$&$4^{16}$&$5^{13}$&$6^{5}$&$99.9739 \%$\\
    \hline
    $30$& {\footnotesize \texttt{deeffag.hbi.jag.ibh.jbi.ibj.ichchbj.hbj.jbgcgdjchcggjDh}}&
      &&&$3^{10}$&$5^{30}$&$99.9612 \%$\\
    \hline
    $31$& {\footnotesize \texttt{deeffaf.gbh.ibjbh.ibh.hbi.i.hbi.jeg.ibgcgdgdifjcjcjbgcitj}}&
      &$3^{7}$&$4^{10}$&$5^{19}$&$6^{5}$&$99.8560 \%$\\
    \hline
    $32$& {\footnotesize \texttt{deefgaf.hbg.ibj.jah.i.jag.jbi.jci.jeigf.jbfgi.hbicj.fjhdhah}}&
      $3^{5}$&$4^{14}$&$5^{18}$&$6^{4}$&$7^{1}$&$99.7397 \%$\\
    \hline
    $33$& {\footnotesize \texttt{deefgaf.hbg.hbi.iah.j.iag.icicj.ibj.jajgf.ibfgj.hbjcjghgfdh.h}}&
      $3^{5}$&$4^{15}$&$5^{16}$&$6^{6}$&$7^{1}$&$99.5714 \%$\\
    \hline
    $34$& {\footnotesize \texttt{deefgaf.hbi.gbh.iahaiag.jbj.jbg.g.hafcg.ibjcf.jbjgcdigjfjhjbccc}}&
      $3^{6}$&$4^{12}$&$5^{19}$&$6^{6}$&$7^{1}$&$99.2457 \%$\\
    \hline
    $35$& {\footnotesize \texttt{deefgaf.hbi.gbh.hajajai.hbgaibj.hbjbjeiafaiciaj.gbjcj.ijghg.ggi.i}}&
      $3^{5}$&$4^{17}$&$5^{16}$&$6^{2}$&$7^{5}$&$99.1755 \%$\\
    \hline
  \end{longtable}
\end{center}
The full list of complexes and their empirical collapsing probabilities for
$5\cdot 10^3$ samples is available from the authors upon request.

\bigskip
Altogether, if we restrict ourselves to $v$-vertex $2$-neighborly $3$-sphere
triangulations, $v\leq 10$, we have for the average empirical collapsing 
probability
\begin{center}
  \begin{tabular}{|r|r|r|}
    \hline
    $v$&$\#$ complexes&$1 - $ avg. coll. prob. \\
    \hline
    \hline
    $\leq 7$& $3$&$0.0000000$\\
    \hline
    $8$& $4$&$0.0000109$\\
    \hline
    $9$& $50$&$0.0002064$\\
    \hline
    $10$& $3540$&$0.0012429$\\
    \hline
  \end{tabular}
\end{center}
This very limited amount of data already exhibits a rapid increase in the 
proportion of non-collapsing sequences for triangulations of increasing size.
This supports speculations that pathologically complicated combinatorial 
objects become rapidly more common as the size of a triangulation increases. 
However, whether or not these numbers actually suggests that the average 
collapsing probability approaches $0$ as $v \to \infty$ remains unclear.

%
%TODO: discuss further results about spanning trees from \cite[Proposition 5]{Aldous90RandomWalkUST} (which are independent of the graph in question):
%
%\begin{itemize}
%  \item Probability of a vertex to be a leaf vertex $\mathbb{P} (v \textrm{ leaf} ) \leq \exp{-3/8} = 0.687289$;
%  \item expectation of a vertex to be a leaf $\mathbb{E} (v \textrm{ leaf} ) \geq \alpha(4)$ where 
%    $$\alpha (r) = \sum \limits_{j=2}^{r-1} r^{-1} (1-\frac{j}{r}) ( 1 + \prod \limits_{i=1}^{j-1} (1 - \frac{1-i/r}{r}). $$
%\end{itemize}
%

\section{Producing nearly non-collapsible $3$-sphere triangulations}
\label{sec:nearlyNonColl}

In this section, we propose a heuristic to pre-evaluate if a 
$3$-sphere triangulation has complicated 
combinatorial properties (such as very few or no collapsing sequences) or not,
based on the simple combinatorial property of the edge degrees of the 
triangulation. Of course, there must be theoretical limits to how effective
this pre-evaluation can be due to the potential hardness of the underlying
problem, but a deeper understanding of the impact of 
simple, local combinatorial structures on (non-)collapsing sequences 
gives valuable insights into triangulations with pathological 
combinatorial characteristics.

\medskip
Let $S$ be a triangulation of the $3$-sphere with $f$-vector 
$f(S) = (v,v+n,2n,n)$ and for any edge $e \in \operatorname{skel}_1 (S)$ 
let $\operatorname{deg}_S (e)$ be the edge degree of $e$ in $S$ (i.e., the
number of tetrahedra of $S$ containing $e$).
Furthermore, let $\Gamma(S)$ be the face pairing graph of $S$ and let 
$T \subset \Gamma(S)$ be a spanning tree.
The $2$-dimensional simplicial complex obtained by collapsing $S$ along $T$ 
is denoted by ${S}_T$, and $e$ is called {\em free} in ${S}_T$ if 
$e$ has degree one in ${S}_T$ (i.e., it is only contained in one triangle of
$S_T$).

For any given spanning tree $T \subset \Gamma(S)$, we expect the number of 
free edges in ${S}_T$ (more precisely, the number of triangles with free edges) 
to strongly correlate with whether or not $T$ leads to a collapsing sequence: 
triangles of ${S}_T$ can be removed as long as there are free edges left and the 
removal of any triangle has a clear tendency to produce new free edges. The 
more free edges there are to begin with, the higher we expect the chances to be
that in this process all triangles can be removed.

In more concrete terms, let $ \operatorname{skel}_1 (S) = \{ 
e_1 , \ldots , e_{n+v} \} $ be the set of edges of $S$ and let
  $$p_i = | \{ e_i \textrm{ free in } S_T \, | \, T \textrm{ spanning 
    tree of } S \} | \,\, /\,\, | \{ \textrm{spanning trees of } S \} | $$
be the probability that edge $e_i$ is free in $S_T$ for $T$ sampled 
uniformly at random. Then, the vector $ (p_1, \ldots , p_{n+v} ) $
contains information about the collapsing probability of $S$.

Obviously, the exact value of the $p_i$ depends on the structure
of $S$ and getting a precise estimate for the $p_i$ involves sampling
spanning trees (which, at the same time, gives us an estimator for the 
collapsing probability -- the quantity we want to pre-evaluate).
Instead we argue that the degree of the edge $e_i$, a quantity which can be 
very easily extracted from $S$, influences the value of $p_i$ by virtue of the
following observation. 

\begin{theorem}
	\label{thm:freedegdedges}
	Let $S$ be a triangulation of the $3$-sphere and let 
  $e_i \in \operatorname{skel}_1 (S)$ be an edge of $S$ of degree 
	$\operatorname{deg}_S (e_i) = k$. Furthermore, let $p_i$ be the proportion of 
  spanning trees of $\Gamma (S)$ for which $e_i$ is free in $S_T$. Then
		$$ p_i \quad \geq \quad \frac{4}{7} \cdot \left ( \frac{4}{13} \right )^{k-2}.$$
\end{theorem}

\begin{proof}
  Since $S$ is a simplicial complex, the star of $e_i$ is represented in 
  $\Gamma (S)$ as the boundary of a $k$-gon $C = \langle \Delta_1 , \Delta_2 , 
  \ldots , \Delta_{k} \rangle$ with no chords. To see this note that
  a chord in the $k$-gon represents an identification of triangular
  boundary faces of the star of $e_i$ in $S$ which contradicts the 
  simplicial complex property (cf. Figure~\ref{fig:spindles} on the left hand 
  side). Moreover, for any spanning tree $T \subset \Gamma (S)$, $e_i$ is 
  free in $S_T$ if and only if $T$ intersects $C$ in a (connected) path of 
  length $k-1$.

  A spanning tree $T \subset \Gamma (S)$ can be found by following a random walk
  in $\Gamma (S)$ discarding all arcs on the way which close a cycle.
  Here, we only concentrate on the probability of one particular class 
  of random walks which always result in a $2$-complex $S_T$ 
  in which $e_i$ is free.
  
  W.l.o.g., let $\Delta_1$ be the first node in $\Gamma (S)$ which is visited
  by the random walk in step $m$. One way for 
  $T \cap C$ to be a path of length $k-1$ is if the random walk travels 
  through all arcs $ (\Delta_j , \Delta_{j+1} )$, $1 \leq j \leq k-1$, 
  travelling back and forth between nodes $\Delta_j$ and $u_j$, $v_j$ or 
  $\Delta_{j-1}$ on the way (cf. Figure~\ref{fig:spindles} on the right hand 
  side).  In step $m+1$ the walk is at one of $\Delta_2$ or $\Delta_k$ with 
  probability $\frac{1}{2}$. If the random walk does not 
  choose one of these two options, there is an overall $\frac{1}{8}$ chance 
  that it revisits $\Delta_1$ in step $m+2$ without visiting any of
  the $\Delta_{\ell}$ first (remember: there are no chords in the $k$-cycle).
  Moreover, there is an at least $\frac{1}{64}$ chance that it
  revisits $\Delta_1$ in step $m+4$, etc. Altogether, there is a
  $$ \sum  \limits_{i \geq 0} 8^{-i} \quad = \quad \frac{4}{7} $$
  chance that the random walk hits $\Delta_2$ or $\Delta_k$ such that
  it is still in the class of random walks we are considering (and, in 
  particular, can still produce a spanning tree leading to a free edge $e_i$).

  \begin{figure}
    \begin{center}
      \includegraphics[width=.8\textwidth]{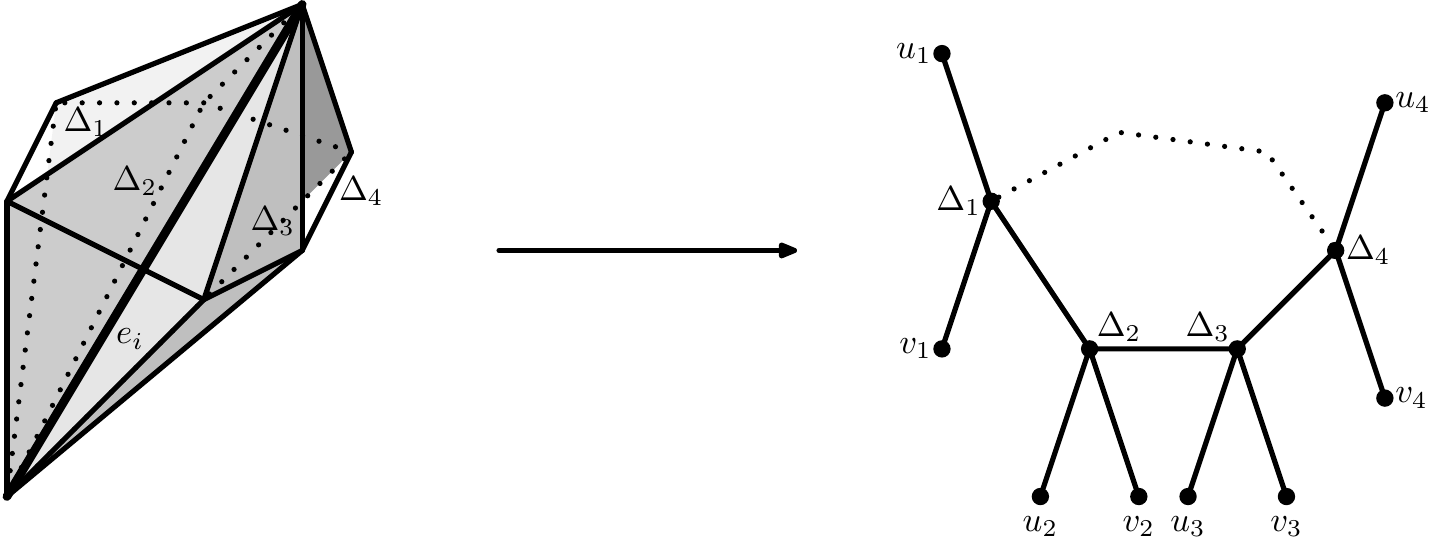}
      \caption{Left: star of an edge in a simplicial $3$-dimensional
          triangulation. Right: the corresponding feature in the 
          dual graph. \label{fig:spindles}}
    \end{center}
  \end{figure}

  W.l.o.g., let the random walk be now at $\Delta_2$ (if it is at 
  $\Delta_k$ we relabel). There is a $\frac{1}{4}$ chance that the random walk
  is at $\Delta_3$ in the next step and a $\frac{3}{16}$ chance
  that the random walk revisits $\Delta_2$ after two steps (without hitting 
  any node other than $u_2$, $v_2$ or $\Delta_1$. Hence, there is an overall
  chance of
  $$ \frac{1}{4} \cdot \sum \limits_{i\geq 0} \left ( \frac{3}{16} \right )^{i} 
    \quad = \quad \frac{4}{13} $$
  that the random walk reaches $\Delta_3$ and still is in the class of 
  random walks we are considering (and, in particular, still has the chance to 
  produce a spanning tree leading to a free edge $e_i$).
  Iterating the argument proves the theorem.
\end{proof}

The above lower bounds are not expected to be sharp, especially not for higher
degrees. However, finding effective upper bounds for the $ p_i $'s is 
difficult due to the unknown global structure of $S$.
Nonetheless, Theorem \ref{thm:freedegdedges} supports the intuitive 
assumption that edges of small degrees have a much higher chance of becoming 
free than higher degree edges. 

\medskip
To see how the lower bounds given in Theorem~\ref{thm:freedegdedges} compare
to the actual values of the $p_i$ we take a closer look at the set of 
$2$-neighbourly $9$-vertex $3$-sphere triangulations. Namely, for
every $9$-vertex $2$-neighbourly $3$-sphere $S$ and for every edge $e_i \in S$,
$1 \leq i \leq n+v$,
we uniformly sample $10^5$ spanning trees $T$ and record how often 
$e_i$ is a free edge in $S_T$ -- the $2$-complex obtained by collapsing all 
tetrahedra of $S$ along $T$. We then compare this number to the degree 
$\operatorname{deg}_S(e_i)$ of $e_i$.

The results of this experiment are shown in Figure~\ref{fig:degExpts}. 
On the horizontal axis all $50$ $2$-neighbourly $9$-vertex triangulations 
of the $3$-sphere are listed. The vertical axis lists the estimators 
$\hat{p}_i$, $1 \leq i \leq n+v$, of the probability of the edge $e_i$ to 
be free in $S_T$ with spanning tree $T$ chosen uniformly at random.
Note how the estimators $\hat{p_i}$ display an exponential decay in the degree 
of the edges, as suggested by Theorem~\ref{thm:freedegdedges}. Moreover,
for most edges $e_i$ of degree $\operatorname{deg}_S(e_i)$, we have for the 
estimator $\hat{p}_i \sim 2^{2-\operatorname{deg}_S(e_i)}$.

\begin{figure}
  \begin{center}
    \input{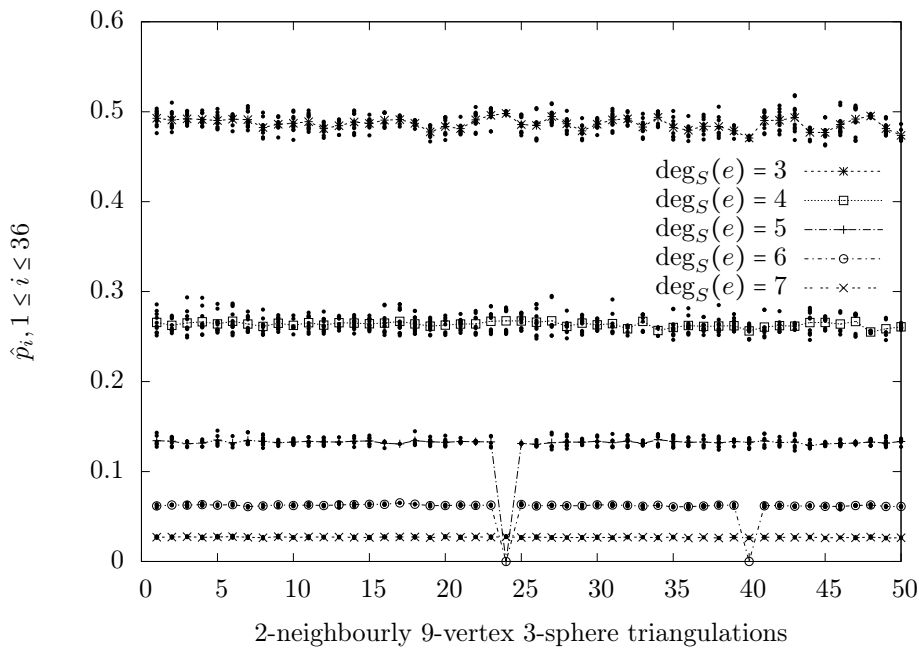}    
    \caption{Every dot represents the estimated probability 
      $\hat{p}_i$ of an edge $e_i$, $1 \leq i \leq 36$, in one of the 
      $2$-neighbourly $9$-vertex $3$-sphere triangulations being free after 
      collapsing the $3$-cells along a uniformly sampled spanning tree. 
      The lines indicate the average probability of an edge being free 
      grouped by degree. Sample size $10^5$. \label{fig:degExpts}}
  \end{center}
\end{figure}

\subsection*{The edge variance of a triangulation}

Assuming the above observations about the quantities $p_i$ hold in reasonable 
generality this motivates the following strategy to produce complicated 
triangulations of the $3$-sphere (i.e., triangulations of the $3$-sphere
with collapsing probability near zero).

\medskip
Let $S$ be an $n$-tetrahedra $v$-vertex triangulation of the $3$-sphere. Every
tetrahedron has six edges and the total number of edges in $S$ is given
by $n+v$. Hence, the average edge degree of an edge in $S$ is given by
$$ \overline{\operatorname{deg}}_{S} = \frac{6 n}{n+v}. $$
In addition, since $S$ is simplicial, every edge degree must be at least three.

We have seen in the above section that, given a spanning tree 
$T \subset \Gamma (S)$ chosen uniformly at random, the value $p_i$ for
edge $e_i$ being of degree $k$ can be assumed to decay exponentially in $k$.
At the same time, it is reasonable to assume that, at least on average, 
collapsing sequences. 

Combining these two statements, we can expect that $3$-sphere
triangulations with few edges of degree three and four should be, on average,
more difficult to collapse than triangulations with many low degree edges.
To quantify this property we define the following combinatorial invariant.

\begin{definition}
  \label{def:edgevar}
  Let $S$ be an $n$-tetrahedron, $v$-vertex simplicial triangulation of 
  the $3$-sphere, and let $\overline{\operatorname{deg}}_{S} = \frac{6 n}{n+v}$
  be its average edge degree. The quantity
    $$ \operatorname{var} (S) = \frac{1}{n+v} \sum \limits_{i=1}^{n+v} 
      (\overline{\operatorname{deg}}_{S} - \operatorname{deg}_S (e_i))^2 $$
  is referred to as the {\it edge variance} of $S$.
\end{definition}

Given a $3$-sphere triangulation $S$, computing $\operatorname{var} (S)$
is a very simple procedure but might give away valuable hints towards
the collapsing probability of $S$. 

Note that this measure must fail in general since non-collapsibility
(i.e., collapsing probability zero) can be a local feature of a
triangulation and thus cannot always be picked up by the 
edge variance (which has a global characteristic). Nonetheless, first experiments 
with heuristics based on the edge variance seem promising (see below).

\begin{remark}
  For experimental evidence that the edge variance is indeed a valuable
  measure of complicatedness, compare the average collapsing probability
  of $8$-, $9$- and $10$-vertex $3$-sphere in Tables~\ref{tab:eight}, 
  \ref{tab:ninemin} and \ref{tab:tenmin},
  and the collapsing probabilities for spheres with minimal edge variance
  in Tables~\ref{tab:eight}, \ref{tab:nine}, and \ref{tab:ten} respectively.
  Keep in mind that some of the complexes admit very few non-collapsing 
  sequences and thus only the estimators of near-neighbourly triangulations
  can be assumed to be reasonably robust.
\end{remark}

\subsection*{A heuristic to produce complicated $3$-sphere triangulations}

The heuristic to produce complicated $3$-sphere triangulations is 
straightforward: Given a $3$-sphere triangulation $S$, we perform bistellar 
one- and two-moves in order to reduce the edge variance.
The heuristic follows a simulated annealing approach where phases of 
reducing the edge variance are followed by phases where the edge variance is
deliberately increased. For a much more detailed discussion of simulated 
annealing type simplification heuristics based on bistellar moves see 
\cite{Bjoerner00SimplMnfBistellarFlips}. The complex with currently the smallest 
edge variance is stored and returned after a maximum number of moves is 
performed.

\medskip
As a proof of concept, we are able to produce a $15$-vertex $3$-sphere $S_{15}$
triangulation with only $2.5903 \pm 0.0618 \%$ collapsing sequences,
error probability $0.01 \%$.
Its facet list is given in Appendix~\ref{app:fiveteen}. See
the table below for a comparison of this number with 
known small and complicated $3$-sphere triangulations from 
\cite{Benedetti13RandomDMT}.

\bigskip
\begin{center}
 \begin{tabular}{|l|l|r|}
  \hline
  triangulation&$f$-vector&exp. coll. prob. ($N = 10^6$) \\
  \hline
  \hline
  $S_{15}$&$(15,105,180,90)$&$0.025903 \pm 0.000618$ \\ 
  \hline
  \texttt{trefoil}&$(13,69,112,56)$&$0.839725 \pm 0.001427$ \\ 
  \hline
  \texttt{double\_trefoil}&$(16,108,184,92)$&$0.193914 \pm 0.001538$ \\ 
  \hline
  \texttt{triple\_trefoil}&$(18,143,250,125)$&$0.000000 \pm 0.000000$ \\ 
  \hline
  \hline
 \end{tabular}
\end{center}
%
%s^2 (N, 839725) = 0.134587
%s/Sqrt(N) = 0.000366861
%z_0.99995 * s/Sqrt(N) = 0.00142731
%
%s^2 (N, 193914) = 0.156312
%s/Sqrt(N) = 0.000395363
%z_0.99995 * s/Sqrt(N) = 0.00153819
%
%s^2 (N, 25903) = 0.0252321
%s/Sqrt(N) = 0.000158846
%z_0.99995 * s/Sqrt(N) = 0.000618005
%

\bigskip
Extending this approach to a more sophisticated heuristic with optimised 
parameters is work in progress.

The potential of this approach lies in applying this framework to
the inverse problem of producing a collapsible triangulation.  
Given a $3$-manifold triangulation $M$, one of the most basic approaches to 
prove that $M$ is a $3$-sphere is to describe a collapsing sequence of 
$M$. This approach, however, cannot work for non-collapsible $3$-sphere
triangulations.
Using the idea of the edge variance, and given a complicated triangulation 
of a $3$-manifold, we first try to increase its edge variance 
and then try to collapse it. 

\medskip
Implementing such a heuristic is simple. However, testing its effectiveness is 
not: as of today there are simply not enough small but complicated $3$-spheres
known to test this approach against existing heuristics.

\section*{Acknowledgement}

This work is supported by CNPq, FAPERJ, PUC-Rio, CAPES, and 
the Department of Industry and Science, Australia under the Australia-India 
Strategic Research Fund (project AISRF06660).

	{\footnotesize
	 \bibliographystyle{abbrv}
	 \bibliography{../../bibliography-combtop/bibliography}
	}

\newpage

\appendix

\section{Non-collapsing sequences of $22$ of the $39$ $3$-sphere triangulations with $8$-vertices}
\label{app:eightV}

\begin{center}
\begin{longtable}{|@{}l@{}|@{}l@{}|@{}l@{}|@{}l@{}|}
  \caption{Non-collapsing sequences (non-perfect discrete Morse functions)
    of $22$ of the $39$ $8$-vertex $3$-sphere triangulations. \label{tab:test}} \\
  \hline
    name&facets&collapsing of $3$-skeleton&contr. non-coll. $2$-compl. \\
  \hline
  \hline
  \endfirsthead

  \multicolumn{4}{l}%
  { \tablename\ \thetable{} -- continued from previous page} \\
  \hline
    name&facets&collapsing of $3$-skeleton&contr. non-coll. $2$-compl. \\
  \endhead

  \hline \multicolumn{4}{r}{{continued on next page --}} \\
  \endfoot

  \endlastfoot

$P_{25}$&
\footnotesize $\begin{array}{lll}
\langle 1,2,3,4 \rangle, & \langle 1,3,6,8 \rangle, & \langle 1,2,3,6 \rangle,\\
\langle 2,3,5,6 \rangle, & \langle 3,5,6,8 \rangle, & \langle 4,5,6,8 \rangle,\\
\langle 3,5,7,8 \rangle, & \langle 4,5,7,8 \rangle, & \langle 4,6,7,8 \rangle,\\
\langle 1,4,5,7 \rangle, & \langle 1,4,6,7 \rangle, & \langle 1,2,4,6 \rangle,\\
\langle 1,6,7,8 \rangle, & \langle 1,3,7,8 \rangle, & \langle 1,3,5,7 \rangle,\\
\langle 1,3,4,5 \rangle, & \langle 2,3,4,5 \rangle, & \langle 2,4,5,6 \rangle.
\end{array}$&
\tiny
$\begin{array}{llllll}
&&\langle 2,4,5,6 \rangle, &
\langle 2,4,5 \rangle&\to&\langle 2,4,5,4 \rangle,\\
\langle 3,4,5 \rangle&\to&\langle 3,4,5,4 \rangle, & \langle 1,3,5 \rangle&\to&\langle 1,3,5,5 \rangle,\\
\langle 1,3,7 \rangle&\to&\langle 1,3,7,7 \rangle, & \langle 1,7,8 \rangle&\to&\langle 1,7,8,7 \rangle,\\
\langle 2,4,6 \rangle&\to&\langle 2,4,6,4 \rangle, & \langle 1,4,6 \rangle&\to&\langle 1,4,6,6 \rangle,\\
\langle 1,4,7 \rangle&\to&\langle 1,4,7,5 \rangle, & \langle 6,7,8 \rangle&\to&\langle 6,7,8,7 \rangle,\\
\langle 4,5,7 \rangle&\to&\langle 4,5,7,7 \rangle, & \langle 5,7,8 \rangle&\to&\langle 5,7,8,7 \rangle,\\
\langle 4,5,8 \rangle&\to&\langle 4,5,8,6 \rangle, & \langle 5,6,8 \rangle&\to&\langle 5,6,8,6 \rangle,\\
\langle 3,5,6 \rangle&\to&\langle 3,5,6,5 \rangle, & \langle 2,3,6 \rangle&\to&\langle 2,3,6,3 \rangle,\\
\langle 1,3,8 \rangle&\to&\langle 1,3,8,6 \rangle, & \langle 1,2,3 \rangle&\to&\langle 1,2,3,3 \rangle.
\end{array}$&
\footnotesize
$\begin{array}{lll}
\langle 4,6,7 \rangle, & \langle 4,7,8 \rangle, & \langle 4,6,8 \rangle,\\
\langle 1,3,6 \rangle, & \langle 1,3,4 \rangle, & \langle 3,6,8 \rangle,\\
\langle 1,6,7 \rangle, & \langle 3,7,8 \rangle, & \langle 2,3,4 \rangle,\\
\langle 1,2,4 \rangle, & \langle 3,5,7 \rangle, & \langle 2,3,5 \rangle,\\
\langle 1,2,6 \rangle, & \langle 2,5,6 \rangle, & \langle 4,5,6 \rangle,\\
\langle 1,4,5 \rangle, & \langle 1,5,7 \rangle.&
\end{array}$\\
\hline\hline

 \end{longtable}
\end{center}

\normalsize

\newpage

\section{A classification of $18$-triangle $8$-vertex triangulations of 
contractible non-collapsible $2$-complexes}
\label{app:18}

\subsection*{Saw-blade complexes}

\subsubsection*{Type I}

\noindent
\begin{tabular}{lllllllll}
$\langle 1, 2, 3 \rangle$, &
$\langle 1, 2, 4 \rangle$, &
$\langle 1, 2, 5 \rangle$, &
$\langle 1, 3, 5 \rangle$, &
$\langle 1, 4, 6 \rangle$, &
$\langle 1, 4, 7 \rangle$, &
$\langle 1, 6, 7 \rangle$, &
$\langle 2, 3, 6 \rangle$, &
$\langle 2, 3, 8 \rangle$, \\

$\langle 2, 4, 7 \rangle$, &
$\langle 2, 5, 7 \rangle$, &
$\langle 2, 6, 8 \rangle$, &
$\langle 3, 4, 5 \rangle$, &
$\langle 3, 4, 6 \rangle$, &
$\langle 3, 4, 8 \rangle$, &
$\langle 4, 5, 8 \rangle$, &
$\langle 5, 6, 7 \rangle$, &
$\langle 5, 6, 8 \rangle$.
\end{tabular}

\bigskip
\hrule
\bigskip

\noindent
\begin{tabular}{lllllllll}
$\langle 1, 2, 3 \rangle$, &
$\langle 1, 2, 4 \rangle$, &
$\langle 1, 2, 5 \rangle$, &
$\langle 1, 3, 5 \rangle$, &
$\langle 1, 4, 6 \rangle$, &
$\langle 1, 4, 7 \rangle$, &
$\langle 1, 6, 7 \rangle$, &
$\langle 2, 3, 6 \rangle$, &
$\langle 2, 3, 8 \rangle$, \\

$\langle 2, 4, 7 \rangle$, &
$\langle 2, 5, 7 \rangle$, &
$\langle 2, 6, 8 \rangle$, &
$\langle 3, 4, 5 \rangle$, &
$\langle 3, 4, 6 \rangle$, &
$\langle 3, 4, 8 \rangle$, &
$\langle 4, 5, 8 \rangle$, &
$\langle 5, 7, 8 \rangle$, &
$\langle 6, 7, 8 \rangle$.
\end{tabular}

\bigskip
\hrule
\bigskip

\noindent
\begin{tabular}{lllllllll}
$\langle 1, 2, 3 \rangle$, &
$\langle 1, 2, 4 \rangle$, &
$\langle 1, 2, 5 \rangle$, &
$\langle 1, 3, 5 \rangle$, &
$\langle 1, 4, 6 \rangle$, &
$\langle 1, 4, 7 \rangle$, &
$\langle 1, 6, 8 \rangle$, &
$\langle 1, 7, 8 \rangle$, &
$\langle 2, 3, 6 \rangle$, \\

$\langle 2, 3, 8 \rangle$, &
$\langle 2, 4, 7 \rangle$, &
$\langle 2, 5, 7 \rangle$, &
$\langle 2, 6, 8 \rangle$, &
$\langle 3, 4, 5 \rangle$, &
$\langle 3, 4, 6 \rangle$, &
$\langle 3, 4, 8 \rangle$, &
$\langle 4, 5, 8 \rangle$, &
$\langle 5, 7, 8 \rangle$.
\end{tabular}

\subsubsection*{Type II}

\noindent
\begin{tabular}{lllllllll}
$\langle 1, 2, 3 \rangle$, &
$\langle 1, 2, 4 \rangle$, &
$\langle 1, 2, 5 \rangle$, &
$\langle 1, 3, 6 \rangle$, &
$\langle 1, 3, 7 \rangle$, &
$\langle 1, 4, 6 \rangle$, &
$\langle 1, 5, 7 \rangle$, &
$\langle 2, 3, 6 \rangle$, &
$\langle 2, 4, 7 \rangle$, \\

$\langle 2, 4, 8 \rangle$, &
$\langle 2, 5, 6 \rangle$, &
$\langle 2, 7, 8 \rangle$, &
$\langle 3, 4, 5 \rangle$, &
$\langle 3, 4, 7 \rangle$, &
$\langle 3, 4, 8 \rangle$, &
$\langle 3, 5, 8 \rangle$, &
$\langle 4, 5, 6 \rangle$, &
$\langle 5, 7, 8 \rangle$.
\end{tabular}

\bigskip
\hrule
\bigskip

\noindent
\begin{tabular}{lllllllll}
$\langle 1, 2, 3 \rangle$, &
$\langle 1, 2, 4 \rangle$, &
$\langle 1, 2, 5 \rangle$, &
$\langle 1, 3, 6 \rangle$, &
$\langle 1, 3, 7 \rangle$, &
$\langle 1, 4, 6 \rangle$, &
$\langle 1, 5, 8 \rangle$, &
$\langle 1, 7, 8 \rangle$, &
$\langle 2, 3, 6 \rangle$, \\

$\langle 2, 4, 7 \rangle$, &
$\langle 2, 4, 8 \rangle$, &
$\langle 2, 5, 6 \rangle$, &
$\langle 2, 7, 8 \rangle$, &
$\langle 3, 4, 5 \rangle$, &
$\langle 3, 4, 7 \rangle$, &
$\langle 3, 4, 8 \rangle$, &
$\langle 3, 5, 8 \rangle$, &
$\langle 4, 5, 6 \rangle$.
\end{tabular}

\subsubsection*{Type III}

\noindent
\begin{tabular}{lllllllll}
$\langle 1, 2, 3 \rangle$, &
$\langle 1, 2, 4 \rangle$, &
$\langle 1, 2, 5 \rangle$, &
$\langle 1, 3, 5 \rangle$, &
$\langle 1, 4, 6 \rangle$, &
$\langle 1, 4, 7 \rangle$, &
$\langle 1, 6, 7 \rangle$, &
$\langle 2, 3, 6 \rangle$, &
$\langle 2, 3, 7 \rangle$, \\

$\langle 2, 4, 6 \rangle$, &
$\langle 2, 5, 7 \rangle$, &
$\langle 3, 4, 5 \rangle$, &
$\langle 3, 4, 7 \rangle$, &
$\langle 3, 4, 8 \rangle$, &
$\langle 3, 6, 8 \rangle$, &
$\langle 4, 5, 8 \rangle$, &
$\langle 5, 6, 7 \rangle$, &
$\langle 5, 6, 8 \rangle$.
\end{tabular}

\bigskip
\hrule
\bigskip

\noindent
\begin{tabular}{lllllllll}
$\langle 1, 2, 3 \rangle$, &
$\langle 1, 2, 4 \rangle$, &
$\langle 1, 2, 5 \rangle$, &
$\langle 1, 3, 5 \rangle$, &
$\langle 1, 4, 6 \rangle$, &
$\langle 1, 4, 7 \rangle$, &
$\langle 1, 6, 7 \rangle$, &
$\langle 2, 3, 6 \rangle$, &
$\langle 2, 3, 7 \rangle$, \\

$\langle 2, 4, 6 \rangle$, &
$\langle 2, 5, 7 \rangle$, &
$\langle 3, 4, 5 \rangle$, &
$\langle 3, 4, 7 \rangle$, &
$\langle 3, 4, 8 \rangle$, &
$\langle 3, 6, 8 \rangle$, &
$\langle 4, 5, 8 \rangle$, &
$\langle 5, 7, 8 \rangle$, &
$\langle 6, 7, 8 \rangle$.
\end{tabular}

\bigskip
\hrule
\bigskip

\noindent
\begin{tabular}{lllllllll}
$\langle 1, 2, 3 \rangle$, &
$\langle 1, 2, 4 \rangle$, &
$\langle 1, 2, 5 \rangle$, &
$\langle 1, 3, 5 \rangle$, &
$\langle 1, 4, 6 \rangle$, &
$\langle 1, 4, 7 \rangle$, &
$\langle 1, 6, 7 \rangle$, &
$\langle 2, 3, 6 \rangle$, &
$\langle 2, 3, 7 \rangle$, \\

$\langle 2, 4, 6 \rangle$, &
$\langle 2, 5, 8 \rangle$, &
$\langle 2, 7, 8 \rangle$, &
$\langle 3, 4, 5 \rangle$, &
$\langle 3, 4, 7 \rangle$, &
$\langle 3, 4, 8 \rangle$, &
$\langle 3, 6, 8 \rangle$, &
$\langle 4, 5, 8 \rangle$, &
$\langle 6, 7, 8 \rangle$.
\end{tabular}

\bigskip
\hrule
\bigskip

\noindent
\begin{tabular}{lllllllll}
$\langle 1, 2, 3 \rangle$, &
$\langle 1, 2, 4 \rangle$, &
$\langle 1, 2, 5 \rangle$, &
$\langle 1, 3, 5 \rangle$, &
$\langle 1, 4, 6 \rangle$, &
$\langle 1, 4, 7 \rangle$, &
$\langle 1, 6, 7 \rangle$, &
$\langle 2, 3, 6 \rangle$, &
$\langle 2, 3, 7 \rangle$, \\

$\langle 2, 4, 8 \rangle$, &
$\langle 2, 5, 8 \rangle$, &
$\langle 2, 6, 7 \rangle$, &
$\langle 3, 4, 5 \rangle$, &
$\langle 3, 4, 6 \rangle$, &
$\langle 3, 4, 8 \rangle$, &
$\langle 3, 7, 8 \rangle$, &
$\langle 4, 5, 7 \rangle$, &
$\langle 5, 7, 8 \rangle$.
\end{tabular}

\bigskip
\hrule
\bigskip

\noindent
\begin{tabular}{lllllllll}
$\langle 1, 2, 3 \rangle$, &
$\langle 1, 2, 4 \rangle$, &
$\langle 1, 2, 5 \rangle$, &
$\langle 1, 3, 5 \rangle$, &
$\langle 1, 4, 6 \rangle$, &
$\langle 1, 4, 7 \rangle$, &
$\langle 1, 6, 7 \rangle$, &
$\langle 2, 3, 6 \rangle$, &
$\langle 2, 3, 8 \rangle$, \\

$\langle 2, 4, 5 \rangle$, &
$\langle 2, 6, 7 \rangle$, &
$\langle 2, 7, 8 \rangle$, &
$\langle 3, 4, 6 \rangle$, &
$\langle 3, 4, 7 \rangle$, &
$\langle 3, 4, 8 \rangle$, &
$\langle 3, 5, 7 \rangle$, &
$\langle 4, 5, 8 \rangle$, &
$\langle 5, 7, 8 \rangle$.
\end{tabular}

\bigskip
\hrule
\bigskip

\noindent
\begin{tabular}{lllllllll}
$\langle 1, 2, 3 \rangle$, &
$\langle 1, 2, 4 \rangle$, &
$\langle 1, 2, 5 \rangle$, &
$\langle 1, 3, 5 \rangle$, &
$\langle 1, 4, 6 \rangle$, &
$\langle 1, 4, 7 \rangle$, &
$\langle 1, 6, 7 \rangle$, &
$\langle 2, 3, 6 \rangle$, &
$\langle 2, 3, 8 \rangle$, \\

$\langle 2, 4, 5 \rangle$, &
$\langle 2, 6, 8 \rangle$, &
$\langle 3, 4, 6 \rangle$, &
$\langle 3, 4, 7 \rangle$, &
$\langle 3, 4, 8 \rangle$, &
$\langle 3, 5, 7 \rangle$, &
$\langle 4, 5, 8 \rangle$, &
$\langle 5, 6, 7 \rangle$, &
$\langle 5, 6, 8 \rangle$.
\end{tabular}

\bigskip
\hrule
\bigskip

\noindent
\begin{tabular}{lllllllll}
$\langle 1, 2, 3 \rangle$, &
$\langle 1, 2, 4 \rangle$, &
$\langle 1, 2, 5 \rangle$, &
$\langle 1, 3, 5 \rangle$, &
$\langle 1, 4, 6 \rangle$, &
$\langle 1, 4, 7 \rangle$, &
$\langle 1, 6, 7 \rangle$, &
$\langle 2, 3, 6 \rangle$, &
$\langle 2, 3, 8 \rangle$, \\

$\langle 2, 4, 5 \rangle$, &
$\langle 2, 6, 8 \rangle$, &
$\langle 3, 4, 6 \rangle$, &
$\langle 3, 4, 7 \rangle$, &
$\langle 3, 4, 8 \rangle$, &
$\langle 3, 5, 7 \rangle$, &
$\langle 4, 5, 8 \rangle$, &
$\langle 5, 7, 8 \rangle$, &
$\langle 6, 7, 8 \rangle$.
\end{tabular}

\bigskip
\hrule
\bigskip

\noindent
\begin{tabular}{lllllllll}
$\langle 1, 2, 3 \rangle$, &
$\langle 1, 2, 4 \rangle$, &
$\langle 1, 2, 5 \rangle$, &
$\langle 1, 3, 5 \rangle$, &
$\langle 1, 4, 6 \rangle$, &
$\langle 1, 4, 7 \rangle$, &
$\langle 1, 6, 7 \rangle$, &
$\langle 2, 3, 6 \rangle$, &
$\langle 2, 3, 8 \rangle$, \\

$\langle 2, 4, 7 \rangle$, &
$\langle 2, 5, 6 \rangle$, &
$\langle 2, 7, 8 \rangle$, &
$\langle 3, 4, 5 \rangle$, &
$\langle 3, 4, 6 \rangle$, &
$\langle 3, 4, 8 \rangle$, &
$\langle 4, 5, 8 \rangle$, &
$\langle 5, 6, 7 \rangle$, &
$\langle 5, 7, 8 \rangle$.
\end{tabular}

\bigskip
\hrule
\bigskip

\noindent
\begin{tabular}{lllllllll}
$\langle 1, 2, 3 \rangle$, &
$\langle 1, 2, 4 \rangle$, &
$\langle 1, 2, 5 \rangle$, &
$\langle 1, 3, 5 \rangle$, &
$\langle 1, 4, 6 \rangle$, &
$\langle 1, 4, 7 \rangle$, &
$\langle 1, 6, 7 \rangle$, &
$\langle 2, 3, 6 \rangle$, &
$\langle 2, 3, 8 \rangle$, \\

$\langle 2, 4, 7 \rangle$, &
$\langle 2, 5, 6 \rangle$, &
$\langle 2, 7, 8 \rangle$, &
$\langle 3, 4, 5 \rangle$, &
$\langle 3, 4, 6 \rangle$, &
$\langle 3, 4, 8 \rangle$, &
$\langle 4, 5, 8 \rangle$, &
$\langle 5, 6, 8 \rangle$, &
$\langle 6, 7, 8 \rangle$.
\end{tabular}

\bigskip
\hrule
\bigskip

\noindent
\begin{tabular}{lllllllll}
$\langle 1, 2, 3 \rangle$, &
$\langle 1, 2, 4 \rangle$, &
$\langle 1, 2, 5 \rangle$, &
$\langle 1, 3, 5 \rangle$, &
$\langle 1, 4, 6 \rangle$, &
$\langle 1, 4, 7 \rangle$, &
$\langle 1, 6, 8 \rangle$, &
$\langle 1, 7, 8 \rangle$, &
$\langle 2, 3, 6 \rangle$, \\

$\langle 2, 3, 7 \rangle$, &
$\langle 2, 4, 6 \rangle$, &
$\langle 2, 5, 7 \rangle$, &
$\langle 3, 4, 5 \rangle$, &
$\langle 3, 4, 7 \rangle$, &
$\langle 3, 4, 8 \rangle$, &
$\langle 3, 6, 8 \rangle$, &
$\langle 4, 5, 8 \rangle$, &
$\langle 5, 7, 8 \rangle$.
\end{tabular}

\bigskip
\hrule
\bigskip

\noindent
\begin{tabular}{lllllllll}
$\langle 1, 2, 3 \rangle$, &
$\langle 1, 2, 4 \rangle$, &
$\langle 1, 2, 5 \rangle$, &
$\langle 1, 3, 5 \rangle$, &
$\langle 1, 4, 6 \rangle$, &
$\langle 1, 4, 7 \rangle$, &
$\langle 1, 6, 8 \rangle$, &
$\langle 1, 7, 8 \rangle$, &
$\langle 2, 3, 6 \rangle$, \\

$\langle 2, 3, 7 \rangle$, &
$\langle 2, 4, 6 \rangle$, &
$\langle 2, 5, 8 \rangle$, &
$\langle 2, 7, 8 \rangle$, &
$\langle 3, 4, 5 \rangle$, &
$\langle 3, 4, 7 \rangle$, &
$\langle 3, 4, 8 \rangle$, &
$\langle 3, 6, 8 \rangle$, &
$\langle 4, 5, 8 \rangle$.
\end{tabular}

\bigskip
\hrule
\bigskip

\noindent
\begin{tabular}{lllllllll}
$\langle 1, 2, 3 \rangle$, &
$\langle 1, 2, 4 \rangle$, &
$\langle 1, 2, 5 \rangle$, &
$\langle 1, 3, 5 \rangle$, &
$\langle 1, 4, 6 \rangle$, &
$\langle 1, 4, 7 \rangle$, &
$\langle 1, 6, 8 \rangle$, &
$\langle 1, 7, 8 \rangle$, &
$\langle 2, 3, 6 \rangle$, \\

$\langle 2, 3, 8 \rangle$, &
$\langle 2, 4, 7 \rangle$, &
$\langle 2, 5, 6 \rangle$, &
$\langle 2, 7, 8 \rangle$, &
$\langle 3, 4, 5 \rangle$, &
$\langle 3, 4, 6 \rangle$, &
$\langle 3, 4, 8 \rangle$, &
$\langle 4, 5, 8 \rangle$, &
$\langle 5, 6, 8 \rangle$.
\end{tabular}

\bigskip
\hrule
\bigskip

\noindent
\begin{tabular}{lllllllll}
$\langle 1, 2, 3 \rangle$, &
$\langle 1, 2, 4 \rangle$, &
$\langle 1, 2, 5 \rangle$, &
$\langle 1, 3, 6 \rangle$, &
$\langle 1, 3, 7 \rangle$, &
$\langle 1, 4, 6 \rangle$, &
$\langle 1, 5, 7 \rangle$, &
$\langle 2, 3, 8 \rangle$, &
$\langle 2, 4, 6 \rangle$, \\

$\langle 2, 4, 7 \rangle$, &
$\langle 2, 5, 8 \rangle$, &
$\langle 2, 6, 7 \rangle$, &
$\langle 3, 4, 5 \rangle$, &
$\langle 3, 4, 7 \rangle$, &
$\langle 3, 4, 8 \rangle$, &
$\langle 3, 5, 6 \rangle$, &
$\langle 4, 5, 8 \rangle$, &
$\langle 5, 6, 7 \rangle$.
\end{tabular}

\bigskip
\hrule
\bigskip

\noindent
\begin{tabular}{lllllllll}
$\langle 1, 2, 3 \rangle$, &
$\langle 1, 2, 5 \rangle$, &
$\langle 1, 2, 6 \rangle$, &
$\langle 1, 3, 5 \rangle$, &
$\langle 1, 4, 6 \rangle$, &
$\langle 1, 4, 7 \rangle$, &
$\langle 1, 4, 8 \rangle$, &
$\langle 1, 7, 8 \rangle$, &
$\langle 2, 3, 7 \rangle$, \\

$\langle 2, 3, 8 \rangle$, &
$\langle 2, 5, 6 \rangle$, &
$\langle 2, 7, 8 \rangle$, &
$\langle 3, 4, 5 \rangle$, &
$\langle 3, 4, 6 \rangle$, &
$\langle 3, 4, 7 \rangle$, &
$\langle 3, 6, 8 \rangle$, &
$\langle 4, 5, 8 \rangle$, &
$\langle 5, 6, 8 \rangle$.
\end{tabular}

\subsection*{Dunce hats}

\noindent
\begin{tabular}{lllllllll}
$\langle 1\,2\,3 \rangle$, &
$\langle 1\,2\,4 \rangle$, &
$\langle 1\,2\,5 \rangle$, &
$\langle 1\,3\,5 \rangle$, &
$\langle 1\,3\,6 \rangle$, &
$\langle 1\,4\,6 \rangle$, &
$\langle 2\,3\,6 \rangle$, &
$\langle 2\,4\,7 \rangle$, &
$\langle 2\,4\,8 \rangle$,\\

$\langle 2\,5\,7 \rangle$, &
$\langle 2\,6\,8 \rangle$, &
$\langle 3\,4\,5 \rangle$, &
$\langle 3\,4\,7 \rangle$, &
$\langle 3\,4\,8 \rangle$, &
$\langle 3\,7\,8 \rangle$, &
$\langle 4\,5\,6 \rangle$, &
$\langle 5\,6\,7 \rangle$, &
$\langle 6\,7\,8 \rangle$.
\end{tabular}

\bigskip
\hrule
\bigskip

\noindent
\begin{tabular}{lllllllll}
$\langle 1\,2\,3 \rangle$, &
$\langle 1\,2\,4 \rangle$, &
$\langle 1\,2\,5 \rangle$, &
$\langle 1\,3\,5 \rangle$, &
$\langle 1\,3\,6 \rangle$, &
$\langle 1\,4\,7 \rangle$, &
$\langle 1\,6\,7 \rangle$, &
$\langle 2\,3\,7 \rangle$, &
$\langle 2\,4\,6 \rangle$,\\

$\langle 2\,4\,7 \rangle$, &
$\langle 2\,5\,6 \rangle$, &
$\langle 3\,4\,5 \rangle$, &
$\langle 3\,4\,6 \rangle$, &
$\langle 3\,4\,8 \rangle$, &
$\langle 3\,7\,8 \rangle$, &
$\langle 4\,5\,8 \rangle$, &
$\langle 5\,6\,7 \rangle$, &
$\langle 5\,7\,8 \rangle$.
\end{tabular}

\bigskip
\hrule
\bigskip

\noindent
\begin{tabular}{lllllllll}
$\langle 1\,2\,3 \rangle$, &
$\langle 1\,2\,4 \rangle$, &
$\langle 1\,2\,5 \rangle$, &
$\langle 1\,3\,5 \rangle$, &
$\langle 1\,3\,6 \rangle$, &
$\langle 1\,4\,7 \rangle$, &
$\langle 1\,6\,7 \rangle$, &
$\langle 2\,3\,7 \rangle$, &
$\langle 2\,4\,6 \rangle$,\\

$\langle 2\,4\,7 \rangle$, &
$\langle 2\,5\,6 \rangle$, &
$\langle 3\,4\,5 \rangle$, &
$\langle 3\,4\,6 \rangle$, &
$\langle 3\,4\,8 \rangle$, &
$\langle 3\,7\,8 \rangle$, &
$\langle 4\,5\,8 \rangle$, &
$\langle 5\,6\,8 \rangle$, &
$\langle 6\,7\,8 \rangle$.
\end{tabular}

\bigskip
\hrule
\bigskip

\noindent
\begin{tabular}{lllllllll}
$\langle 1\,2\,3 \rangle$, &
$\langle 1\,2\,4 \rangle$, &
$\langle 1\,2\,5 \rangle$, &
$\langle 1\,3\,5 \rangle$, &
$\langle 1\,3\,6 \rangle$, &
$\langle 1\,4\,7 \rangle$, &
$\langle 1\,6\,7 \rangle$, &
$\langle 2\,3\,7 \rangle$, &
$\langle 2\,4\,6 \rangle$,\\

$\langle 2\,4\,7 \rangle$, &
$\langle 2\,5\,8 \rangle$, &
$\langle 2\,6\,8 \rangle$, &
$\langle 3\,4\,5 \rangle$, &
$\langle 3\,4\,6 \rangle$, &
$\langle 3\,4\,8 \rangle$, &
$\langle 3\,7\,8 \rangle$, &
$\langle 4\,5\,8 \rangle$, &
$\langle 6\,7\,8 \rangle$.
\end{tabular}

\bigskip
\hrule
\bigskip

\noindent
\begin{tabular}{lllllllll}
$\langle 1\,2\,3 \rangle$, &
$\langle 1\,2\,4 \rangle$, &
$\langle 1\,2\,5 \rangle$, &
$\langle 1\,3\,5 \rangle$, &
$\langle 1\,3\,6 \rangle$, &
$\langle 1\,4\,7 \rangle$, &
$\langle 1\,6\,7 \rangle$, &
$\langle 2\,3\,7 \rangle$, &
$\langle 2\,4\,6 \rangle$,\\

$\langle 2\,4\,8 \rangle$, &
$\langle 2\,5\,7 \rangle$, &
$\langle 2\,6\,8 \rangle$, &
$\langle 3\,4\,5 \rangle$, &
$\langle 3\,4\,6 \rangle$, &
$\langle 3\,4\,7 \rangle$, &
$\langle 4\,5\,8 \rangle$, &
$\langle 5\,6\,7 \rangle$, &
$\langle 5\,6\,8 \rangle$.
\end{tabular}

\bigskip
\hrule
\bigskip

\noindent
\begin{tabular}{lllllllll}
$\langle 1\,2\,3 \rangle$, &
$\langle 1\,2\,4 \rangle$, &
$\langle 1\,2\,5 \rangle$, &
$\langle 1\,3\,5 \rangle$, &
$\langle 1\,3\,6 \rangle$, &
$\langle 1\,4\,7 \rangle$, &
$\langle 1\,6\,7 \rangle$, &
$\langle 2\,3\,7 \rangle$, &
$\langle 2\,4\,6 \rangle$,\\

$\langle 2\,4\,8 \rangle$, &
$\langle 2\,5\,7 \rangle$, &
$\langle 2\,6\,8 \rangle$, &
$\langle 3\,4\,5 \rangle$, &
$\langle 3\,4\,6 \rangle$, &
$\langle 3\,4\,7 \rangle$, &
$\langle 4\,5\,8 \rangle$, &
$\langle 5\,7\,8 \rangle$, &
$\langle 6\,7\,8 \rangle$.
\end{tabular}

\bigskip
\hrule
\bigskip

\noindent
\begin{tabular}{lllllllll}
$\langle 1\,2\,3 \rangle$, &
$\langle 1\,2\,4 \rangle$, &
$\langle 1\,2\,5 \rangle$, &
$\langle 1\,3\,5 \rangle$, &
$\langle 1\,3\,6 \rangle$, &
$\langle 1\,4\,7 \rangle$, &
$\langle 1\,6\,7 \rangle$, &
$\langle 2\,3\,7 \rangle$, &
$\langle 2\,4\,6 \rangle$,\\

$\langle 2\,4\,8 \rangle$, &
$\langle 2\,5\,7 \rangle$, &
$\langle 2\,6\,8 \rangle$, &
$\langle 3\,4\,5 \rangle$, &
$\langle 3\,4\,6 \rangle$, &
$\langle 3\,4\,8 \rangle$, &
$\langle 3\,7\,8 \rangle$, &
$\langle 4\,5\,7 \rangle$, &
$\langle 6\,7\,8 \rangle$.
\end{tabular}

\bigskip
\hrule
\bigskip

\noindent
\begin{tabular}{lllllllll}
$\langle 1\,2\,3 \rangle$, &
$\langle 1\,2\,4 \rangle$, &
$\langle 1\,2\,5 \rangle$, &
$\langle 1\,3\,5 \rangle$, &
$\langle 1\,3\,6 \rangle$, &
$\langle 1\,4\,7 \rangle$, &
$\langle 1\,6\,7 \rangle$, &
$\langle 2\,3\,7 \rangle$, &
$\langle 2\,4\,6 \rangle$,\\

$\langle 2\,4\,8 \rangle$, &
$\langle 2\,5\,7 \rangle$, &
$\langle 2\,6\,8 \rangle$, &
$\langle 3\,4\,5 \rangle$, &
$\langle 3\,4\,7 \rangle$, &
$\langle 3\,4\,8 \rangle$, &
$\langle 3\,6\,8 \rangle$, &
$\langle 4\,5\,6 \rangle$, &
$\langle 5\,6\,7 \rangle$.
\end{tabular}

\bigskip
\hrule
\bigskip

\noindent
\begin{tabular}{lllllllll}
$\langle 1\,2\,3 \rangle$, &
$\langle 1\,2\,4 \rangle$, &
$\langle 1\,2\,5 \rangle$, &
$\langle 1\,3\,5 \rangle$, &
$\langle 1\,3\,6 \rangle$, &
$\langle 1\,4\,7 \rangle$, &
$\langle 1\,6\,7 \rangle$, &
$\langle 2\,3\,7 \rangle$, &
$\langle 2\,4\,6 \rangle$,\\

$\langle 2\,4\,8 \rangle$, &
$\langle 2\,5\,7 \rangle$, &
$\langle 2\,6\,8 \rangle$, &
$\langle 3\,4\,6 \rangle$, &
$\langle 3\,4\,7 \rangle$, &
$\langle 3\,4\,8 \rangle$, &
$\langle 3\,5\,8 \rangle$, &
$\langle 5\,6\,7 \rangle$, &
$\langle 5\,6\,8 \rangle$.
\end{tabular}

\bigskip
\hrule
\bigskip

\noindent
\begin{tabular}{lllllllll}
$\langle 1\,2\,3 \rangle$, &
$\langle 1\,2\,4 \rangle$, &
$\langle 1\,2\,5 \rangle$, &
$\langle 1\,3\,5 \rangle$, &
$\langle 1\,3\,6 \rangle$, &
$\langle 1\,4\,7 \rangle$, &
$\langle 1\,6\,7 \rangle$, &
$\langle 2\,3\,7 \rangle$, &
$\langle 2\,4\,6 \rangle$,\\

$\langle 2\,4\,8 \rangle$, &
$\langle 2\,5\,7 \rangle$, &
$\langle 2\,6\,8 \rangle$, &
$\langle 3\,4\,6 \rangle$, &
$\langle 3\,4\,7 \rangle$, &
$\langle 3\,4\,8 \rangle$, &
$\langle 3\,5\,8 \rangle$, &
$\langle 5\,7\,8 \rangle$, &
$\langle 6\,7\,8 \rangle$.
\end{tabular}

\bigskip
\hrule
\bigskip

\noindent
\begin{tabular}{lllllllll}
$\langle 1\,2\,3 \rangle$, &
$\langle 1\,2\,4 \rangle$, &
$\langle 1\,2\,5 \rangle$, &
$\langle 1\,3\,5 \rangle$, &
$\langle 1\,3\,6 \rangle$, &
$\langle 1\,4\,7 \rangle$, &
$\langle 1\,6\,7 \rangle$, &
$\langle 2\,3\,7 \rangle$, &
$\langle 2\,4\,7 \rangle$,\\

$\langle 2\,4\,8 \rangle$, &
$\langle 2\,5\,6 \rangle$, &
$\langle 2\,6\,8 \rangle$, &
$\langle 3\,4\,5 \rangle$, &
$\langle 3\,4\,6 \rangle$, &
$\langle 3\,4\,8 \rangle$, &
$\langle 3\,7\,8 \rangle$, &
$\langle 4\,5\,6 \rangle$, &
$\langle 6\,7\,8 \rangle$.
\end{tabular}

\bigskip
\hrule
\bigskip

\noindent
\begin{tabular}{lllllllll}
$\langle 1\,2\,3 \rangle$, &
$\langle 1\,2\,4 \rangle$, &
$\langle 1\,2\,5 \rangle$, &
$\langle 1\,3\,5 \rangle$, &
$\langle 1\,3\,6 \rangle$, &
$\langle 1\,4\,6 \rangle$, &
$\langle 2\,3\,6 \rangle$, &
$\langle 2\,4\,7 \rangle$, &
$\langle 2\,4\,8 \rangle$,\\

$\langle 2\,5\,7 \rangle$, &
$\langle 2\,6\,8 \rangle$, &
$\langle 3\,4\,5 \rangle$, &
$\langle 3\,4\,7 \rangle$, &
$\langle 3\,4\,8 \rangle$, &
$\langle 3\,7\,8 \rangle$, &
$\langle 4\,5\,6 \rangle$, &
$\langle 5\,6\,8 \rangle$, &
$\langle 5\,7\,8 \rangle$.
\end{tabular}

\bigskip
\hrule
\bigskip

\noindent
\begin{tabular}{lllllllll}
$\langle 1\,2\,3 \rangle$, &
$\langle 1\,2\,4 \rangle$, &
$\langle 1\,2\,5 \rangle$, &
$\langle 1\,3\,5 \rangle$, &
$\langle 1\,3\,6 \rangle$, &
$\langle 1\,4\,7 \rangle$, &
$\langle 1\,6\,7 \rangle$, &
$\langle 2\,3\,7 \rangle$, &
$\langle 2\,4\,7 \rangle$,\\

$\langle 2\,4\,8 \rangle$, &
$\langle 2\,5\,8 \rangle$, &
$\langle 3\,4\,5 \rangle$, &
$\langle 3\,4\,6 \rangle$, &
$\langle 3\,4\,8 \rangle$, &
$\langle 3\,7\,8 \rangle$, &
$\langle 4\,5\,6 \rangle$, &
$\langle 5\,6\,7 \rangle$, &
$\langle 5\,7\,8 \rangle$.
\end{tabular}

\bigskip
\hrule
\bigskip

\noindent
\begin{tabular}{lllllllll}
$\langle 1\,2\,3 \rangle$, &
$\langle 1\,2\,4 \rangle$, &
$\langle 1\,2\,5 \rangle$, &
$\langle 1\,3\,5 \rangle$, &
$\langle 1\,3\,6 \rangle$, &
$\langle 1\,4\,7 \rangle$, &
$\langle 1\,6\,7 \rangle$, &
$\langle 2\,3\,7 \rangle$, &
$\langle 2\,4\,7 \rangle$,\\

$\langle 2\,4\,8 \rangle$, &
$\langle 2\,5\,8 \rangle$, &
$\langle 3\,4\,5 \rangle$, &
$\langle 3\,4\,6 \rangle$, &
$\langle 3\,4\,8 \rangle$, &
$\langle 3\,7\,8 \rangle$, &
$\langle 4\,5\,6 \rangle$, &
$\langle 5\,6\,8 \rangle$, &
$\langle 6\,7\,8 \rangle$.
\end{tabular}

\bigskip
\hrule
\bigskip

\noindent
\begin{tabular}{lllllllll}
$\langle 1\,2\,3 \rangle$, &
$\langle 1\,2\,4 \rangle$, &
$\langle 1\,2\,5 \rangle$, &
$\langle 1\,3\,5 \rangle$, &
$\langle 1\,3\,6 \rangle$, &
$\langle 1\,4\,7 \rangle$, &
$\langle 1\,6\,7 \rangle$, &
$\langle 2\,3\,8 \rangle$, &
$\langle 2\,4\,6 \rangle$,\\

$\langle 2\,4\,7 \rangle$, &
$\langle 2\,5\,6 \rangle$, &
$\langle 2\,7\,8 \rangle$, &
$\langle 3\,4\,5 \rangle$, &
$\langle 3\,4\,6 \rangle$, &
$\langle 3\,4\,8 \rangle$, &
$\langle 4\,5\,8 \rangle$, &
$\langle 5\,6\,7 \rangle$, &
$\langle 5\,7\,8 \rangle$.
\end{tabular}

\bigskip
\hrule
\bigskip

\noindent
\begin{tabular}{lllllllll}
$\langle 1\,2\,3 \rangle$, &
$\langle 1\,2\,4 \rangle$, &
$\langle 1\,2\,5 \rangle$, &
$\langle 1\,3\,5 \rangle$, &
$\langle 1\,3\,6 \rangle$, &
$\langle 1\,4\,7 \rangle$, &
$\langle 1\,6\,7 \rangle$, &
$\langle 2\,3\,8 \rangle$, &
$\langle 2\,4\,6 \rangle$,\\

$\langle 2\,4\,7 \rangle$, &
$\langle 2\,5\,6 \rangle$, &
$\langle 2\,7\,8 \rangle$, &
$\langle 3\,4\,5 \rangle$, &
$\langle 3\,4\,6 \rangle$, &
$\langle 3\,4\,8 \rangle$, &
$\langle 4\,5\,8 \rangle$, &
$\langle 5\,6\,8 \rangle$, &
$\langle 6\,7\,8 \rangle$.
\end{tabular}

\bigskip
\hrule
\bigskip

\noindent
\begin{tabular}{lllllllll}
$\langle 1\,2\,3 \rangle$, &
$\langle 1\,2\,4 \rangle$, &
$\langle 1\,2\,5 \rangle$, &
$\langle 1\,3\,5 \rangle$, &
$\langle 1\,3\,6 \rangle$, &
$\langle 1\,4\,7 \rangle$, &
$\langle 1\,6\,7 \rangle$, &
$\langle 2\,3\,8 \rangle$, &
$\langle 2\,4\,6 \rangle$,\\

$\langle 2\,4\,7 \rangle$, &
$\langle 2\,5\,7 \rangle$, &
$\langle 2\,6\,8 \rangle$, &
$\langle 3\,4\,5 \rangle$, &
$\langle 3\,4\,6 \rangle$, &
$\langle 3\,4\,8 \rangle$, &
$\langle 4\,5\,8 \rangle$, &
$\langle 5\,6\,7 \rangle$, &
$\langle 5\,6\,8 \rangle$.
\end{tabular}

\bigskip
\hrule
\bigskip

\noindent
\begin{tabular}{lllllllll}
$\langle 1\,2\,3 \rangle$, &
$\langle 1\,2\,4 \rangle$, &
$\langle 1\,2\,5 \rangle$, &
$\langle 1\,3\,5 \rangle$, &
$\langle 1\,3\,6 \rangle$, &
$\langle 1\,4\,7 \rangle$, &
$\langle 1\,6\,7 \rangle$, &
$\langle 2\,3\,8 \rangle$, &
$\langle 2\,4\,6 \rangle$,\\

$\langle 2\,4\,7 \rangle$, &
$\langle 2\,5\,7 \rangle$, &
$\langle 2\,6\,8 \rangle$, &
$\langle 3\,4\,5 \rangle$, &
$\langle 3\,4\,6 \rangle$, &
$\langle 3\,4\,8 \rangle$, &
$\langle 4\,5\,8 \rangle$, &
$\langle 5\,7\,8 \rangle$, &
$\langle 6\,7\,8 \rangle$.
\end{tabular}

\bigskip
\hrule
\bigskip

\noindent
\begin{tabular}{lllllllll}
$\langle 1\,2\,3 \rangle$, &
$\langle 1\,2\,4 \rangle$, &
$\langle 1\,2\,5 \rangle$, &
$\langle 1\,3\,5 \rangle$, &
$\langle 1\,3\,6 \rangle$, &
$\langle 1\,4\,7 \rangle$, &
$\langle 1\,6\,8 \rangle$, &
$\langle 1\,7\,8 \rangle$, &
$\langle 2\,3\,6 \rangle$,\\

$\langle 2\,4\,6 \rangle$, &
$\langle 2\,4\,8 \rangle$, &
$\langle 2\,5\,8 \rangle$, &
$\langle 3\,4\,5 \rangle$, &
$\langle 3\,4\,7 \rangle$, &
$\langle 3\,4\,8 \rangle$, &
$\langle 3\,7\,8 \rangle$, &
$\langle 4\,5\,6 \rangle$, &
$\langle 5\,6\,8 \rangle$.
\end{tabular}

\bigskip
\hrule
\bigskip

\noindent
\begin{tabular}{lllllllll}
$\langle 1\,2\,3 \rangle$, &
$\langle 1\,2\,4 \rangle$, &
$\langle 1\,2\,5 \rangle$, &
$\langle 1\,3\,5 \rangle$, &
$\langle 1\,3\,6 \rangle$, &
$\langle 1\,4\,7 \rangle$, &
$\langle 1\,6\,8 \rangle$, &
$\langle 1\,7\,8 \rangle$, &
$\langle 2\,3\,7 \rangle$,\\

$\langle 2\,4\,6 \rangle$, &
$\langle 2\,4\,7 \rangle$, &
$\langle 2\,5\,6 \rangle$, &
$\langle 3\,4\,5 \rangle$, &
$\langle 3\,4\,6 \rangle$, &
$\langle 3\,4\,8 \rangle$, &
$\langle 3\,7\,8 \rangle$, &
$\langle 4\,5\,8 \rangle$, &
$\langle 5\,6\,8 \rangle$.
\end{tabular}

\bigskip
\hrule
\bigskip

\noindent
\begin{tabular}{lllllllll}
$\langle 1\,2\,3 \rangle$, &
$\langle 1\,2\,4 \rangle$, &
$\langle 1\,2\,5 \rangle$, &
$\langle 1\,3\,5 \rangle$, &
$\langle 1\,3\,6 \rangle$, &
$\langle 1\,4\,7 \rangle$, &
$\langle 1\,6\,8 \rangle$, &
$\langle 1\,7\,8 \rangle$, &
$\langle 2\,3\,7 \rangle$,\\

$\langle 2\,4\,6 \rangle$, &
$\langle 2\,4\,7 \rangle$, &
$\langle 2\,5\,8 \rangle$, &
$\langle 2\,6\,8 \rangle$, &
$\langle 3\,4\,5 \rangle$, &
$\langle 3\,4\,6 \rangle$, &
$\langle 3\,4\,8 \rangle$, &
$\langle 3\,7\,8 \rangle$, &
$\langle 4\,5\,8 \rangle$.
\end{tabular}

\bigskip
\hrule
\bigskip

\noindent
\begin{tabular}{lllllllll}
$\langle 1\,2\,3 \rangle$, &
$\langle 1\,2\,4 \rangle$, &
$\langle 1\,2\,5 \rangle$, &
$\langle 1\,3\,5 \rangle$, &
$\langle 1\,3\,6 \rangle$, &
$\langle 1\,4\,7 \rangle$, &
$\langle 1\,6\,8 \rangle$, &
$\langle 1\,7\,8 \rangle$, &
$\langle 2\,3\,7 \rangle$,\\

$\langle 2\,4\,6 \rangle$, &
$\langle 2\,4\,8 \rangle$, &
$\langle 2\,5\,7 \rangle$, &
$\langle 2\,6\,8 \rangle$, &
$\langle 3\,4\,5 \rangle$, &
$\langle 3\,4\,6 \rangle$, &
$\langle 3\,4\,7 \rangle$, &
$\langle 4\,5\,8 \rangle$, &
$\langle 5\,7\,8 \rangle$.
\end{tabular}

\bigskip
\hrule
\bigskip

\noindent
\begin{tabular}{lllllllll}
$\langle 1\,2\,3 \rangle$, &
$\langle 1\,2\,4 \rangle$, &
$\langle 1\,2\,5 \rangle$, &
$\langle 1\,3\,5 \rangle$, &
$\langle 1\,3\,6 \rangle$, &
$\langle 1\,4\,6 \rangle$, &
$\langle 2\,3\,7 \rangle$, &
$\langle 2\,4\,6 \rangle$, &
$\langle 2\,4\,8 \rangle$,\\

$\langle 2\,5\,6 \rangle$, &
$\langle 2\,7\,8 \rangle$, &
$\langle 3\,4\,5 \rangle$, &
$\langle 3\,4\,7 \rangle$, &
$\langle 3\,4\,8 \rangle$, &
$\langle 3\,6\,8 \rangle$, &
$\langle 4\,5\,7 \rangle$, &
$\langle 5\,6\,7 \rangle$, &
$\langle 6\,7\,8 \rangle$.
\end{tabular}

\bigskip
\hrule
\bigskip

\noindent
\begin{tabular}{lllllllll}
$\langle 1\,2\,3 \rangle$, &
$\langle 1\,2\,4 \rangle$, &
$\langle 1\,2\,5 \rangle$, &
$\langle 1\,3\,5 \rangle$, &
$\langle 1\,3\,6 \rangle$, &
$\langle 1\,4\,7 \rangle$, &
$\langle 1\,6\,8 \rangle$, &
$\langle 1\,7\,8 \rangle$, &
$\langle 2\,3\,7 \rangle$,\\

$\langle 2\,4\,6 \rangle$, &
$\langle 2\,4\,8 \rangle$, &
$\langle 2\,5\,7 \rangle$, &
$\langle 2\,6\,8 \rangle$, &
$\langle 3\,4\,5 \rangle$, &
$\langle 3\,4\,6 \rangle$, &
$\langle 3\,4\,8 \rangle$, &
$\langle 3\,7\,8 \rangle$, &
$\langle 4\,5\,7 \rangle$.
\end{tabular}

\bigskip
\hrule
\bigskip

\noindent
\begin{tabular}{lllllllll}
$\langle 1\,2\,3 \rangle$, &
$\langle 1\,2\,4 \rangle$, &
$\langle 1\,2\,5 \rangle$, &
$\langle 1\,3\,5 \rangle$, &
$\langle 1\,3\,6 \rangle$, &
$\langle 1\,4\,7 \rangle$, &
$\langle 1\,6\,8 \rangle$, &
$\langle 1\,7\,8 \rangle$, &
$\langle 2\,3\,7 \rangle$,\\

$\langle 2\,4\,6 \rangle$, &
$\langle 2\,4\,8 \rangle$, &
$\langle 2\,5\,7 \rangle$, &
$\langle 2\,6\,8 \rangle$, &
$\langle 3\,4\,6 \rangle$, &
$\langle 3\,4\,7 \rangle$, &
$\langle 3\,4\,8 \rangle$, &
$\langle 3\,5\,8 \rangle$, &
$\langle 5\,7\,8 \rangle$.
\end{tabular}

\bigskip
\hrule
\bigskip

\noindent
\begin{tabular}{lllllllll}
$\langle 1\,2\,3 \rangle$, &
$\langle 1\,2\,4 \rangle$, &
$\langle 1\,2\,5 \rangle$, &
$\langle 1\,3\,5 \rangle$, &
$\langle 1\,3\,6 \rangle$, &
$\langle 1\,4\,7 \rangle$, &
$\langle 1\,6\,8 \rangle$, &
$\langle 1\,7\,8 \rangle$, &
$\langle 2\,3\,7 \rangle$,\\

$\langle 2\,4\,7 \rangle$, &
$\langle 2\,4\,8 \rangle$, &
$\langle 2\,5\,6 \rangle$, &
$\langle 2\,6\,8 \rangle$, &
$\langle 3\,4\,5 \rangle$, &
$\langle 3\,4\,6 \rangle$, &
$\langle 3\,4\,8 \rangle$, &
$\langle 3\,7\,8 \rangle$, &
$\langle 4\,5\,6 \rangle$.
\end{tabular}

\bigskip
\hrule
\bigskip

\noindent
\begin{tabular}{lllllllll}
$\langle 1\,2\,3 \rangle$, &
$\langle 1\,2\,4 \rangle$, &
$\langle 1\,2\,5 \rangle$, &
$\langle 1\,3\,5 \rangle$, &
$\langle 1\,3\,6 \rangle$, &
$\langle 1\,4\,7 \rangle$, &
$\langle 1\,6\,8 \rangle$, &
$\langle 1\,7\,8 \rangle$, &
$\langle 2\,3\,7 \rangle$,\\

$\langle 2\,4\,7 \rangle$, &
$\langle 2\,4\,8 \rangle$, &
$\langle 2\,5\,8 \rangle$, &
$\langle 3\,4\,5 \rangle$, &
$\langle 3\,4\,6 \rangle$, &
$\langle 3\,4\,8 \rangle$, &
$\langle 3\,7\,8 \rangle$, &
$\langle 4\,5\,6 \rangle$, &
$\langle 5\,6\,8 \rangle$.
\end{tabular}

\bigskip
\hrule
\bigskip

\noindent
\begin{tabular}{lllllllll}
$\langle 1\,2\,3 \rangle$, &
$\langle 1\,2\,4 \rangle$, &
$\langle 1\,2\,5 \rangle$, &
$\langle 1\,3\,5 \rangle$, &
$\langle 1\,3\,6 \rangle$, &
$\langle 1\,4\,7 \rangle$, &
$\langle 1\,6\,8 \rangle$, &
$\langle 1\,7\,8 \rangle$, &
$\langle 2\,3\,8 \rangle$,\\

$\langle 2\,4\,6 \rangle$, &
$\langle 2\,4\,7 \rangle$, &
$\langle 2\,5\,6 \rangle$, &
$\langle 2\,7\,8 \rangle$, &
$\langle 3\,4\,5 \rangle$, &
$\langle 3\,4\,6 \rangle$, &
$\langle 3\,4\,8 \rangle$, &
$\langle 4\,5\,8 \rangle$, &
$\langle 5\,6\,8 \rangle$.
\end{tabular}

\bigskip
\hrule
\bigskip

\noindent
\begin{tabular}{lllllllll}
$\langle 1\,2\,3 \rangle$, &
$\langle 1\,2\,4 \rangle$, &
$\langle 1\,2\,5 \rangle$, &
$\langle 1\,3\,5 \rangle$, &
$\langle 1\,3\,6 \rangle$, &
$\langle 1\,4\,7 \rangle$, &
$\langle 1\,6\,8 \rangle$, &
$\langle 1\,7\,8 \rangle$, &
$\langle 2\,3\,8 \rangle$,\\

$\langle 2\,4\,6 \rangle$, &
$\langle 2\,4\,7 \rangle$, &
$\langle 2\,5\,7 \rangle$, &
$\langle 2\,6\,8 \rangle$, &
$\langle 3\,4\,5 \rangle$, &
$\langle 3\,4\,6 \rangle$, &
$\langle 3\,4\,8 \rangle$, &
$\langle 4\,5\,8 \rangle$, &
$\langle 5\,7\,8 \rangle$.
\end{tabular}

\bigskip
\hrule
\bigskip

\noindent
\begin{tabular}{lllllllll}
$\langle 1\,2\,3 \rangle$, &
$\langle 1\,2\,4 \rangle$, &
$\langle 1\,2\,5 \rangle$, &
$\langle 1\,3\,5 \rangle$, &
$\langle 1\,4\,6 \rangle$, &
$\langle 1\,4\,7 \rangle$, &
$\langle 1\,6\,7 \rangle$, &
$\langle 2\,3\,6 \rangle$, &
$\langle 2\,3\,7 \rangle$,\\

$\langle 2\,4\,8 \rangle$, &
$\langle 2\,5\,6 \rangle$, &
$\langle 2\,7\,8 \rangle$, &
$\langle 3\,4\,5 \rangle$, &
$\langle 3\,4\,6 \rangle$, &
$\langle 3\,4\,7 \rangle$, &
$\langle 4\,5\,8 \rangle$, &
$\langle 5\,6\,7 \rangle$, &
$\langle 5\,7\,8 \rangle$.
\end{tabular}

\bigskip
\hrule
\bigskip

\noindent
\begin{tabular}{lllllllll}
$\langle 1\,2\,3 \rangle$, &
$\langle 1\,2\,4 \rangle$, &
$\langle 1\,2\,5 \rangle$, &
$\langle 1\,3\,5 \rangle$, &
$\langle 1\,4\,6 \rangle$, &
$\langle 1\,4\,7 \rangle$, &
$\langle 1\,6\,7 \rangle$, &
$\langle 2\,3\,6 \rangle$, &
$\langle 2\,3\,7 \rangle$,\\

$\langle 2\,4\,8 \rangle$, &
$\langle 2\,5\,6 \rangle$, &
$\langle 2\,7\,8 \rangle$, &
$\langle 3\,4\,5 \rangle$, &
$\langle 3\,4\,6 \rangle$, &
$\langle 3\,4\,7 \rangle$, &
$\langle 4\,5\,8 \rangle$, &
$\langle 5\,6\,8 \rangle$, &
$\langle 6\,7\,8 \rangle$.
\end{tabular}

\bigskip
\hrule
\bigskip

\noindent
\begin{tabular}{lllllllll}
$\langle 1\,2\,3 \rangle$, &
$\langle 1\,2\,4 \rangle$, &
$\langle 1\,2\,5 \rangle$, &
$\langle 1\,3\,5 \rangle$, &
$\langle 1\,4\,6 \rangle$, &
$\langle 1\,4\,7 \rangle$, &
$\langle 1\,6\,7 \rangle$, &
$\langle 2\,3\,6 \rangle$, &
$\langle 2\,3\,7 \rangle$,\\

$\langle 2\,4\,8 \rangle$, &
$\langle 2\,5\,6 \rangle$, &
$\langle 2\,7\,8 \rangle$, &
$\langle 3\,4\,5 \rangle$, &
$\langle 3\,4\,7 \rangle$, &
$\langle 3\,4\,8 \rangle$, &
$\langle 3\,6\,8 \rangle$, &
$\langle 4\,5\,6 \rangle$, &
$\langle 6\,7\,8 \rangle$.
\end{tabular}

\bigskip
\hrule
\bigskip

\noindent
\begin{tabular}{lllllllll}
$\langle 1\,2\,3 \rangle$, &
$\langle 1\,2\,4 \rangle$, &
$\langle 1\,2\,5 \rangle$, &
$\langle 1\,3\,5 \rangle$, &
$\langle 1\,4\,6 \rangle$, &
$\langle 1\,4\,7 \rangle$, &
$\langle 1\,6\,7 \rangle$, &
$\langle 2\,3\,6 \rangle$, &
$\langle 2\,3\,7 \rangle$,\\

$\langle 2\,4\,8 \rangle$, &
$\langle 2\,5\,6 \rangle$, &
$\langle 2\,7\,8 \rangle$, &
$\langle 3\,4\,6 \rangle$, &
$\langle 3\,4\,7 \rangle$, &
$\langle 3\,4\,8 \rangle$, &
$\langle 3\,5\,8 \rangle$, &
$\langle 5\,6\,7 \rangle$, &
$\langle 5\,7\,8 \rangle$.
\end{tabular}

\bigskip
\hrule
\bigskip

\noindent
\begin{tabular}{lllllllll}
$\langle 1\,2\,3 \rangle$, &
$\langle 1\,2\,4 \rangle$, &
$\langle 1\,2\,5 \rangle$, &
$\langle 1\,3\,5 \rangle$, &
$\langle 1\,3\,6 \rangle$, &
$\langle 1\,4\,6 \rangle$, &
$\langle 2\,3\,7 \rangle$, &
$\langle 2\,4\,6 \rangle$, &
$\langle 2\,4\,8 \rangle$,\\

$\langle 2\,5\,6 \rangle$, &
$\langle 2\,7\,8 \rangle$, &
$\langle 3\,4\,5 \rangle$, &
$\langle 3\,4\,7 \rangle$, &
$\langle 3\,4\,8 \rangle$, &
$\langle 3\,6\,8 \rangle$, &
$\langle 4\,5\,7 \rangle$, &
$\langle 5\,6\,8 \rangle$, &
$\langle 5\,7\,8 \rangle$.
\end{tabular}

\bigskip
\hrule
\bigskip

\noindent
\begin{tabular}{lllllllll}
$\langle 1\,2\,3 \rangle$, &
$\langle 1\,2\,4 \rangle$, &
$\langle 1\,2\,5 \rangle$, &
$\langle 1\,3\,5 \rangle$, &
$\langle 1\,4\,6 \rangle$, &
$\langle 1\,4\,7 \rangle$, &
$\langle 1\,6\,7 \rangle$, &
$\langle 2\,3\,6 \rangle$, &
$\langle 2\,3\,7 \rangle$,\\

$\langle 2\,4\,8 \rangle$, &
$\langle 2\,5\,6 \rangle$, &
$\langle 2\,7\,8 \rangle$, &
$\langle 3\,4\,6 \rangle$, &
$\langle 3\,4\,7 \rangle$, &
$\langle 3\,4\,8 \rangle$, &
$\langle 3\,5\,8 \rangle$, &
$\langle 5\,6\,8 \rangle$, &
$\langle 6\,7\,8 \rangle$.
\end{tabular}

\bigskip
\hrule
\bigskip

\noindent
\begin{tabular}{lllllllll}
$\langle 1\,2\,3 \rangle$, &
$\langle 1\,2\,4 \rangle$, &
$\langle 1\,2\,5 \rangle$, &
$\langle 1\,3\,5 \rangle$, &
$\langle 1\,4\,6 \rangle$, &
$\langle 1\,4\,7 \rangle$, &
$\langle 1\,6\,7 \rangle$, &
$\langle 2\,3\,6 \rangle$, &
$\langle 2\,3\,8 \rangle$,\\

$\langle 2\,4\,8 \rangle$, &
$\langle 2\,5\,6 \rangle$, &
$\langle 3\,4\,5 \rangle$, &
$\langle 3\,4\,6 \rangle$, &
$\langle 3\,4\,7 \rangle$, &
$\langle 3\,7\,8 \rangle$, &
$\langle 4\,5\,8 \rangle$, &
$\langle 5\,6\,7 \rangle$, &
$\langle 5\,7\,8 \rangle$.
\end{tabular}

\bigskip
\hrule
\bigskip

\noindent
\begin{tabular}{lllllllll}
$\langle 1\,2\,3 \rangle$, &
$\langle 1\,2\,4 \rangle$, &
$\langle 1\,2\,5 \rangle$, &
$\langle 1\,3\,5 \rangle$, &
$\langle 1\,4\,6 \rangle$, &
$\langle 1\,4\,7 \rangle$, &
$\langle 1\,6\,7 \rangle$, &
$\langle 2\,3\,6 \rangle$, &
$\langle 2\,3\,8 \rangle$,\\

$\langle 2\,4\,8 \rangle$, &
$\langle 2\,5\,6 \rangle$, &
$\langle 3\,4\,5 \rangle$, &
$\langle 3\,4\,6 \rangle$, &
$\langle 3\,4\,7 \rangle$, &
$\langle 3\,7\,8 \rangle$, &
$\langle 4\,5\,8 \rangle$, &
$\langle 5\,6\,8 \rangle$, &
$\langle 6\,7\,8 \rangle$.
\end{tabular}

\bigskip
\hrule
\bigskip

\noindent
\begin{tabular}{lllllllll}
$\langle 1\,2\,3 \rangle$, &
$\langle 1\,2\,4 \rangle$, &
$\langle 1\,2\,5 \rangle$, &
$\langle 1\,3\,5 \rangle$, &
$\langle 1\,4\,6 \rangle$, &
$\langle 1\,4\,7 \rangle$, &
$\langle 1\,6\,7 \rangle$, &
$\langle 2\,3\,6 \rangle$, &
$\langle 2\,3\,8 \rangle$,\\

$\langle 2\,4\,8 \rangle$, &
$\langle 2\,5\,7 \rangle$, &
$\langle 2\,6\,7 \rangle$, &
$\langle 3\,4\,5 \rangle$, &
$\langle 3\,4\,6 \rangle$, &
$\langle 3\,4\,7 \rangle$, &
$\langle 3\,7\,8 \rangle$, &
$\langle 4\,5\,8 \rangle$, &
$\langle 5\,7\,8 \rangle$.
\end{tabular}

\bigskip
\hrule
\bigskip

\noindent
\begin{tabular}{lllllllll}
$\langle 1\,2\,3 \rangle$, &
$\langle 1\,2\,4 \rangle$, &
$\langle 1\,2\,5 \rangle$, &
$\langle 1\,3\,5 \rangle$, &
$\langle 1\,4\,6 \rangle$, &
$\langle 1\,4\,7 \rangle$, &
$\langle 1\,6\,8 \rangle$, &
$\langle 1\,7\,8 \rangle$, &
$\langle 2\,3\,6 \rangle$,\\

$\langle 2\,3\,7 \rangle$, &
$\langle 2\,4\,8 \rangle$, &
$\langle 2\,5\,6 \rangle$, &
$\langle 2\,7\,8 \rangle$, &
$\langle 3\,4\,5 \rangle$, &
$\langle 3\,4\,6 \rangle$, &
$\langle 3\,4\,7 \rangle$, &
$\langle 4\,5\,8 \rangle$, &
$\langle 5\,6\,8 \rangle$.
\end{tabular}

\bigskip
\hrule
\bigskip

\noindent
\begin{tabular}{lllllllll}
$\langle 1\,2\,3 \rangle$, &
$\langle 1\,2\,4 \rangle$, &
$\langle 1\,2\,5 \rangle$, &
$\langle 1\,3\,5 \rangle$, &
$\langle 1\,4\,6 \rangle$, &
$\langle 1\,4\,7 \rangle$, &
$\langle 1\,6\,8 \rangle$, &
$\langle 1\,7\,8 \rangle$, &
$\langle 2\,3\,6 \rangle$,\\

$\langle 2\,3\,7 \rangle$, &
$\langle 2\,4\,8 \rangle$, &
$\langle 2\,5\,6 \rangle$, &
$\langle 2\,7\,8 \rangle$, &
$\langle 3\,4\,5 \rangle$, &
$\langle 3\,4\,7 \rangle$, &
$\langle 3\,4\,8 \rangle$, &
$\langle 3\,6\,8 \rangle$, &
$\langle 4\,5\,6 \rangle$.
\end{tabular}

\bigskip
\hrule
\bigskip

\noindent
\begin{tabular}{lllllllll}
$\langle 1\,2\,3 \rangle$, &
$\langle 1\,2\,4 \rangle$, &
$\langle 1\,2\,5 \rangle$, &
$\langle 1\,3\,5 \rangle$, &
$\langle 1\,4\,6 \rangle$, &
$\langle 1\,4\,7 \rangle$, &
$\langle 1\,6\,8 \rangle$, &
$\langle 1\,7\,8 \rangle$, &
$\langle 2\,3\,6 \rangle$,\\

$\langle 2\,3\,7 \rangle$, &
$\langle 2\,4\,8 \rangle$, &
$\langle 2\,5\,6 \rangle$, &
$\langle 2\,7\,8 \rangle$, &
$\langle 3\,4\,6 \rangle$, &
$\langle 3\,4\,7 \rangle$, &
$\langle 3\,4\,8 \rangle$, &
$\langle 3\,5\,8 \rangle$, &
$\langle 5\,6\,8 \rangle$.
\end{tabular}

\bigskip
\hrule
\bigskip

\noindent
\begin{tabular}{lllllllll}
$\langle 1\,2\,3 \rangle$, &
$\langle 1\,2\,4 \rangle$, &
$\langle 1\,2\,5 \rangle$, &
$\langle 1\,3\,5 \rangle$, &
$\langle 1\,4\,6 \rangle$, &
$\langle 1\,4\,7 \rangle$, &
$\langle 1\,6\,8 \rangle$, &
$\langle 1\,7\,8 \rangle$, &
$\langle 2\,3\,6 \rangle$,\\

$\langle 2\,3\,8 \rangle$, &
$\langle 2\,4\,8 \rangle$, &
$\langle 2\,5\,6 \rangle$, &
$\langle 3\,4\,5 \rangle$, &
$\langle 3\,4\,6 \rangle$, &
$\langle 3\,4\,7 \rangle$, &
$\langle 3\,7\,8 \rangle$, &
$\langle 4\,5\,8 \rangle$, &
$\langle 5\,6\,8 \rangle$.
\end{tabular}

\bigskip
\hrule
\bigskip

\noindent
\begin{tabular}{lllllllll}
$\langle 1\,2\,3 \rangle$, &
$\langle 1\,2\,4 \rangle$, &
$\langle 1\,2\,5 \rangle$, &
$\langle 1\,3\,6 \rangle$, &
$\langle 1\,3\,7 \rangle$, &
$\langle 1\,4\,6 \rangle$, &
$\langle 1\,5\,7 \rangle$, &
$\langle 2\,3\,6 \rangle$, &
$\langle 2\,4\,7 \rangle$,\\

$\langle 2\,4\,8 \rangle$, &
$\langle 2\,5\,8 \rangle$, &
$\langle 2\,6\,7 \rangle$, &
$\langle 3\,4\,5 \rangle$, &
$\langle 3\,4\,7 \rangle$, &
$\langle 3\,4\,8 \rangle$, &
$\langle 3\,5\,8 \rangle$, &
$\langle 4\,5\,6 \rangle$, &
$\langle 5\,6\,7 \rangle$.
\end{tabular}

\bigskip
\hrule
\bigskip

\noindent
\begin{tabular}{lllllllll}
$\langle 1\,2\,3 \rangle$, &
$\langle 1\,2\,4 \rangle$, &
$\langle 1\,2\,5 \rangle$, &
$\langle 1\,3\,6 \rangle$, &
$\langle 1\,3\,7 \rangle$, &
$\langle 1\,4\,6 \rangle$, &
$\langle 1\,5\,7 \rangle$, &
$\langle 2\,3\,8 \rangle$, &
$\langle 2\,4\,6 \rangle$,\\

$\langle 2\,4\,7 \rangle$, &
$\langle 2\,5\,6 \rangle$, &
$\langle 2\,7\,8 \rangle$, &
$\langle 3\,4\,5 \rangle$, &
$\langle 3\,4\,7 \rangle$, &
$\langle 3\,4\,8 \rangle$, &
$\langle 3\,5\,6 \rangle$, &
$\langle 4\,5\,8 \rangle$, &
$\langle 5\,7\,8 \rangle$.
\end{tabular}

\bigskip
\hrule
\bigskip

\noindent
\begin{tabular}{lllllllll}
$\langle 1\,2\,3 \rangle$, &
$\langle 1\,2\,4 \rangle$, &
$\langle 1\,2\,5 \rangle$, &
$\langle 1\,3\,5 \rangle$, &
$\langle 1\,3\,6 \rangle$, &
$\langle 1\,4\,6 \rangle$, &
$\langle 2\,3\,7 \rangle$, &
$\langle 2\,4\,6 \rangle$, &
$\langle 2\,4\,8 \rangle$,\\

$\langle 2\,5\,8 \rangle$, &
$\langle 2\,6\,7 \rangle$, &
$\langle 3\,4\,5 \rangle$, &
$\langle 3\,4\,7 \rangle$, &
$\langle 3\,4\,8 \rangle$, &
$\langle 3\,6\,8 \rangle$, &
$\langle 4\,5\,7 \rangle$, &
$\langle 5\,6\,7 \rangle$, &
$\langle 5\,6\,8 \rangle$.
\end{tabular}

\bigskip
\hrule
\bigskip

\noindent
\begin{tabular}{lllllllll}
$\langle 1\,2\,3 \rangle$, &
$\langle 1\,2\,4 \rangle$, &
$\langle 1\,2\,5 \rangle$, &
$\langle 1\,3\,6 \rangle$, &
$\langle 1\,3\,7 \rangle$, &
$\langle 1\,4\,6 \rangle$, &
$\langle 1\,5\,8 \rangle$, &
$\langle 1\,7\,8 \rangle$, &
$\langle 2\,3\,8 \rangle$,\\

$\langle 2\,4\,6 \rangle$, &
$\langle 2\,4\,7 \rangle$, &
$\langle 2\,5\,6 \rangle$, &
$\langle 2\,7\,8 \rangle$, &
$\langle 3\,4\,5 \rangle$, &
$\langle 3\,4\,7 \rangle$, &
$\langle 3\,4\,8 \rangle$, &
$\langle 3\,5\,6 \rangle$, &
$\langle 4\,5\,8 \rangle$.
\end{tabular}

\bigskip
\hrule
\bigskip

\noindent
\begin{tabular}{lllllllll}
$\langle 1\,2\,3 \rangle$, &
$\langle 1\,2\,4 \rangle$, &
$\langle 1\,2\,5 \rangle$, &
$\langle 1\,3\,6 \rangle$, &
$\langle 1\,3\,7 \rangle$, &
$\langle 1\,4\,8 \rangle$, &
$\langle 1\,5\,6 \rangle$, &
$\langle 1\,7\,8 \rangle$, &
$\langle 2\,3\,8 \rangle$,\\

$\langle 2\,4\,6 \rangle$, &
$\langle 2\,4\,7 \rangle$, &
$\langle 2\,5\,8 \rangle$, &
$\langle 2\,6\,7 \rangle$, &
$\langle 3\,4\,6 \rangle$, &
$\langle 3\,4\,7 \rangle$, &
$\langle 3\,4\,8 \rangle$, &
$\langle 5\,6\,7 \rangle$, &
$\langle 5\,7\,8 \rangle$.
\end{tabular}

\bigskip
\hrule
\bigskip

\noindent
\begin{tabular}{lllllllll}
$\langle 1\,2\,3 \rangle$, &
$\langle 1\,2\,4 \rangle$, &
$\langle 1\,2\,5 \rangle$, &
$\langle 1\,3\,6 \rangle$, &
$\langle 1\,3\,7 \rangle$, &
$\langle 1\,4\,8 \rangle$, &
$\langle 1\,5\,6 \rangle$, &
$\langle 1\,7\,8 \rangle$, &
$\langle 2\,3\,8 \rangle$,\\

$\langle 2\,4\,6 \rangle$, &
$\langle 2\,4\,7 \rangle$, &
$\langle 2\,5\,8 \rangle$, &
$\langle 2\,6\,7 \rangle$, &
$\langle 3\,4\,6 \rangle$, &
$\langle 3\,4\,7 \rangle$, &
$\langle 3\,4\,8 \rangle$, &
$\langle 5\,6\,8 \rangle$, &
$\langle 6\,7\,8 \rangle$.
\end{tabular}

\bigskip
\hrule
\bigskip

\noindent
\begin{tabular}{lllllllll}
$\langle 1\,2\,3 \rangle$, &
$\langle 1\,2\,5 \rangle$, &
$\langle 1\,2\,6 \rangle$, &
$\langle 1\,3\,5 \rangle$, &
$\langle 1\,3\,7 \rangle$, &
$\langle 1\,6\,7 \rangle$, &
$\langle 2\,3\,6 \rangle$, &
$\langle 2\,4\,5 \rangle$, &
$\langle 2\,4\,7 \rangle$,\\

$\langle 2\,4\,8 \rangle$, &
$\langle 2\,7\,8 \rangle$, &
$\langle 3\,4\,6 \rangle$, &
$\langle 3\,4\,7 \rangle$, &
$\langle 3\,4\,8 \rangle$, &
$\langle 3\,5\,8 \rangle$, &
$\langle 4\,5\,6 \rangle$, &
$\langle 5\,6\,7 \rangle$, &
$\langle 5\,7\,8 \rangle$.
\end{tabular}

\bigskip
\hrule
\bigskip

\noindent
\begin{tabular}{lllllllll}
$\langle 1\,2\,3 \rangle$, &
$\langle 1\,2\,5 \rangle$, &
$\langle 1\,2\,6 \rangle$, &
$\langle 1\,3\,5 \rangle$, &
$\langle 1\,3\,7 \rangle$, &
$\langle 1\,6\,7 \rangle$, &
$\langle 2\,3\,6 \rangle$, &
$\langle 2\,4\,5 \rangle$, &
$\langle 2\,4\,7 \rangle$,\\

$\langle 2\,4\,8 \rangle$, &
$\langle 2\,7\,8 \rangle$, &
$\langle 3\,4\,6 \rangle$, &
$\langle 3\,4\,7 \rangle$, &
$\langle 3\,4\,8 \rangle$, &
$\langle 3\,5\,8 \rangle$, &
$\langle 4\,5\,6 \rangle$, &
$\langle 5\,6\,8 \rangle$, &
$\langle 6\,7\,8 \rangle$.
\end{tabular}

\bigskip
\hrule
\bigskip

\noindent
\begin{tabular}{lllllllll}
$\langle 1\,2\,3 \rangle$, &
$\langle 1\,2\,5 \rangle$, &
$\langle 1\,2\,6 \rangle$, &
$\langle 1\,3\,5 \rangle$, &
$\langle 1\,3\,7 \rangle$, &
$\langle 1\,6\,8 \rangle$, &
$\langle 1\,7\,8 \rangle$, &
$\langle 2\,3\,6 \rangle$, &
$\langle 2\,4\,5 \rangle$,\\

$\langle 2\,4\,7 \rangle$, &
$\langle 2\,4\,8 \rangle$, &
$\langle 2\,7\,8 \rangle$, &
$\langle 3\,4\,6 \rangle$, &
$\langle 3\,4\,7 \rangle$, &
$\langle 3\,4\,8 \rangle$, &
$\langle 3\,5\,8 \rangle$, &
$\langle 4\,5\,6 \rangle$, &
$\langle 5\,6\,8 \rangle$.
\end{tabular}

\bigskip
\hrule
\bigskip

\noindent
\begin{tabular}{lllllllll}
$\langle 1\,2\,3 \rangle$, &
$\langle 1\,2\,5 \rangle$, &
$\langle 1\,2\,6 \rangle$, &
$\langle 1\,3\,5 \rangle$, &
$\langle 1\,4\,6 \rangle$, &
$\langle 1\,4\,7 \rangle$, &
$\langle 1\,4\,8 \rangle$, &
$\langle 1\,7\,8 \rangle$, &
$\langle 2\,3\,7 \rangle$,\\

$\langle 2\,3\,8 \rangle$, &
$\langle 2\,4\,5 \rangle$, &
$\langle 2\,4\,7 \rangle$, &
$\langle 2\,6\,8 \rangle$, &
$\langle 3\,4\,5 \rangle$, &
$\langle 3\,4\,6 \rangle$, &
$\langle 3\,4\,8 \rangle$, &
$\langle 3\,6\,7 \rangle$, &
$\langle 6\,7\,8 \rangle$.
\end{tabular}

\bigskip
\hrule
\bigskip

\noindent
\begin{tabular}{lllllllll}
$\langle 1\,2\,3 \rangle$, &
$\langle 1\,2\,5 \rangle$, &
$\langle 1\,2\,6 \rangle$, &
$\langle 1\,3\,5 \rangle$, &
$\langle 1\,4\,6 \rangle$, &
$\langle 1\,4\,7 \rangle$, &
$\langle 1\,4\,8 \rangle$, &
$\langle 1\,7\,8 \rangle$, &
$\langle 2\,3\,7 \rangle$,\\

$\langle 2\,3\,8 \rangle$, &
$\langle 2\,5\,7 \rangle$, &
$\langle 2\,6\,8 \rangle$, &
$\langle 3\,4\,5 \rangle$, &
$\langle 3\,4\,6 \rangle$, &
$\langle 3\,4\,8 \rangle$, &
$\langle 3\,6\,7 \rangle$, &
$\langle 4\,5\,7 \rangle$, &
$\langle 6\,7\,8 \rangle$.
\end{tabular}

\bigskip
\hrule
\bigskip

\noindent
\begin{tabular}{lllllllll}
$\langle 1\,2\,3 \rangle$, &
$\langle 1\,2\,5 \rangle$, &
$\langle 1\,2\,6 \rangle$, &
$\langle 1\,3\,5 \rangle$, &
$\langle 1\,4\,6 \rangle$, &
$\langle 1\,4\,7 \rangle$, &
$\langle 1\,4\,8 \rangle$, &
$\langle 1\,7\,8 \rangle$, &
$\langle 2\,3\,7 \rangle$,\\

$\langle 2\,3\,8 \rangle$, &
$\langle 2\,5\,7 \rangle$, &
$\langle 2\,6\,8 \rangle$, &
$\langle 3\,4\,5 \rangle$, &
$\langle 3\,4\,7 \rangle$, &
$\langle 3\,4\,8 \rangle$, &
$\langle 4\,5\,6 \rangle$, &
$\langle 5\,6\,7 \rangle$, &
$\langle 6\,7\,8 \rangle$.
\end{tabular}

\bigskip
\hrule
\bigskip

\noindent
\begin{tabular}{lllllllll}
$\langle 1\,2\,3 \rangle$, &
$\langle 1\,2\,5 \rangle$, &
$\langle 1\,2\,6 \rangle$, &
$\langle 1\,3\,5 \rangle$, &
$\langle 1\,4\,6 \rangle$, &
$\langle 1\,4\,7 \rangle$, &
$\langle 1\,4\,8 \rangle$, &
$\langle 1\,7\,8 \rangle$, &
$\langle 2\,3\,7 \rangle$,\\

$\langle 2\,3\,8 \rangle$, &
$\langle 2\,5\,7 \rangle$, &
$\langle 2\,6\,8 \rangle$, &
$\langle 3\,4\,5 \rangle$, &
$\langle 3\,4\,7 \rangle$, &
$\langle 3\,4\,8 \rangle$, &
$\langle 4\,5\,6 \rangle$, &
$\langle 5\,6\,8 \rangle$, &
$\langle 5\,7\,8 \rangle$.
\end{tabular}

\bigskip
\hrule
\bigskip

\noindent
\begin{tabular}{lllllllll}
$\langle 1\,2\,3 \rangle$, &
$\langle 1\,2\,4 \rangle$, &
$\langle 1\,2\,5 \rangle$, &
$\langle 1\,3\,5 \rangle$, &
$\langle 1\,3\,6 \rangle$, &
$\langle 1\,4\,6 \rangle$, &
$\langle 2\,3\,7 \rangle$, &
$\langle 2\,4\,6 \rangle$, &
$\langle 2\,4\,8 \rangle$,\\

$\langle 2\,5\,8 \rangle$, &
$\langle 2\,6\,7 \rangle$, &
$\langle 3\,4\,5 \rangle$, &
$\langle 3\,4\,7 \rangle$, &
$\langle 3\,4\,8 \rangle$, &
$\langle 3\,6\,8 \rangle$, &
$\langle 4\,5\,7 \rangle$, &
$\langle 5\,7\,8 \rangle$, &
$\langle 6\,7\,8 \rangle$.
\end{tabular}

\bigskip
\hrule
\bigskip

\noindent
\begin{tabular}{lllllllll}
$\langle 1\,2\,3 \rangle$, &
$\langle 1\,2\,5 \rangle$, &
$\langle 1\,2\,6 \rangle$, &
$\langle 1\,3\,5 \rangle$, &
$\langle 1\,4\,6 \rangle$, &
$\langle 1\,4\,7 \rangle$, &
$\langle 1\,4\,8 \rangle$, &
$\langle 1\,7\,8 \rangle$, &
$\langle 2\,3\,7 \rangle$,\\

$\langle 2\,3\,8 \rangle$, &
$\langle 2\,5\,7 \rangle$, &
$\langle 2\,6\,8 \rangle$, &
$\langle 3\,4\,6 \rangle$, &
$\langle 3\,4\,7 \rangle$, &
$\langle 3\,4\,8 \rangle$, &
$\langle 3\,5\,6 \rangle$, &
$\langle 5\,6\,7 \rangle$, &
$\langle 6\,7\,8 \rangle$.
\end{tabular}

\bigskip
\hrule
\bigskip

\noindent
\begin{tabular}{lllllllll}
$\langle 1\,2\,3 \rangle$, &
$\langle 1\,2\,5 \rangle$, &
$\langle 1\,2\,6 \rangle$, &
$\langle 1\,3\,5 \rangle$, &
$\langle 1\,4\,6 \rangle$, &
$\langle 1\,4\,7 \rangle$, &
$\langle 1\,4\,8 \rangle$, &
$\langle 1\,7\,8 \rangle$, &
$\langle 2\,3\,7 \rangle$,\\

$\langle 2\,3\,8 \rangle$, &
$\langle 2\,5\,7 \rangle$, &
$\langle 2\,6\,8 \rangle$, &
$\langle 3\,4\,6 \rangle$, &
$\langle 3\,4\,7 \rangle$, &
$\langle 3\,4\,8 \rangle$, &
$\langle 3\,5\,6 \rangle$, &
$\langle 5\,6\,8 \rangle$, &
$\langle 5\,7\,8 \rangle$.
\end{tabular}

\bigskip
\hrule
\bigskip

\noindent
\begin{tabular}{lllllllll}
$\langle 1\,2\,3 \rangle$, &
$\langle 1\,2\,4 \rangle$, &
$\langle 1\,2\,5 \rangle$, &
$\langle 1\,3\,5 \rangle$, &
$\langle 1\,3\,6 \rangle$, &
$\langle 1\,4\,7 \rangle$, &
$\langle 1\,6\,7 \rangle$, &
$\langle 2\,3\,6 \rangle$, &
$\langle 2\,4\,6 \rangle$,\\

$\langle 2\,4\,8 \rangle$, &
$\langle 2\,5\,7 \rangle$, &
$\langle 2\,7\,8 \rangle$, &
$\langle 3\,4\,5 \rangle$, &
$\langle 3\,4\,7 \rangle$, &
$\langle 3\,4\,8 \rangle$, &
$\langle 3\,7\,8 \rangle$, &
$\langle 4\,5\,6 \rangle$, &
$\langle 5\,6\,7 \rangle$.
\end{tabular}

\bigskip
\hrule
\bigskip

\noindent
\begin{tabular}{lllllllll}
$\langle 1\,2\,3 \rangle$, &
$\langle 1\,2\,4 \rangle$, &
$\langle 1\,2\,5 \rangle$, &
$\langle 1\,3\,5 \rangle$, &
$\langle 1\,3\,6 \rangle$, &
$\langle 1\,4\,7 \rangle$, &
$\langle 1\,6\,7 \rangle$, &
$\langle 2\,3\,6 \rangle$, &
$\langle 2\,4\,6 \rangle$,\\

$\langle 2\,4\,8 \rangle$, &
$\langle 2\,5\,8 \rangle$, &
$\langle 3\,4\,5 \rangle$, &
$\langle 3\,4\,7 \rangle$, &
$\langle 3\,4\,8 \rangle$, &
$\langle 3\,7\,8 \rangle$, &
$\langle 4\,5\,6 \rangle$, &
$\langle 5\,6\,7 \rangle$, &
$\langle 5\,7\,8 \rangle$.
\end{tabular}%

\newpage

\section{A complicated $15$-vertex triangulation of the $3$-sphere}
\label{app:fiveteen}

\noindent
$f$-vector: $(15, 105, 180, 90)$ \\

\noindent
Collapsing probability: $2.5903 \pm 0.0618 \%$  (for error probability $0.01 \%$)

\bigskip
\noindent
Isomorphism signature: \\

\noindent
\texttt{deefgaf.hbi.gbj.kalamai.nbo.n.lbl.oaj.hbg.k.oaoamaj.ibi.hbl.hbm.obg} \\
\texttt{.l.kam.jcicn.mak.nbkakcmak.lahanbm.nbg.mclcmbncn.jbnchambkbccc.jbfa} \\
\texttt{kaf.fcifkbiekgmbchikm.nbnfboddf.iecsjanOd} \\

\bigskip
\begin{center}
%
\end{center}

\end{document}